\documentclass{article}

\usepackage{arxiv}

\usepackage[utf8]{inputenc} % allow utf-8 input
\usepackage[T1]{fontenc}    % use 8-bit T1 fonts
\usepackage{hyperref}       % hyperlinks
\usepackage{url}            % simple URL typesetting
\usepackage{booktabs}       % professional-quality tables
\usepackage{amsfonts}       % blackboard math symbols
\usepackage{nicefrac}       % compact symbols for 1/2, etc.
\usepackage{microtype}      % microtypography
\usepackage{lipsum}
\usepackage{graphicx}
\usepackage{subfigure}
\usepackage{amsmath}
\usepackage{amsthm}
\usepackage{algorithm}
\usepackage[noend]{algpseudocode}
\usepackage[dvipsnames]{xcolor}
\usepackage{enumerate}
\usepackage{comment}

\DeclareMathOperator*{\argmin}{argmin}

\newtheorem{thm}{Theorem}
\newtheorem{lem}{Lemma}
\newtheorem{cor}{Corollary}

\newcommand{\M}{\mathcal{M}}

\title{Minimum Number of Bends of Paths of Trees in a Grid Embedding}

\author{
  Vitor T. F. de Luca\\
  Universidade do Estado do Rio de Janeiro\\
  Rio de Janeiro, Brazil\\
  \texttt{toccivitor8@gmail.com} \\
   \And
 Fabiano S. Oliveira \\
  Universidade do Estado do Rio de Janeiro\\
  Rio de Janeiro, Brazil\\
  \texttt{fabiano.oliveira@ime.uerj.br} \\
  \AND
  Jayme L. Szwarcfiter \\
  Universidade Federal do Rio de Janeiro \hspace{0.5cm} Universidade do Estado do Rio de Janeiro \\
  Rio de Janeiro, Brazil \\
  \texttt{jayme@nce.ufrj.br} \\
}

\begin{document}
\maketitle

\begin{abstract}
We are interested in embedding trees $T$ with $\Delta(T) \leq 4$ in a rectangular grid, such that the vertices of $T$ correspond to grid points, while edges of $T$ correspond to non-intersecting straight segments of the grid lines. Such embeddings are called straight models. While each edge is represented by a straight segment, a path of $T$ is represented in the model by the union of the segments corresponding to its edges, which may consist of a path in the model having several bends. The aim is to determine a straight model of a given tree $T$ minimizing the maximum number of bends over all paths of $T$. We provide a quadratic-time algorithm for this problem. We also show how to construct straight models that have $k$ as its minimum number of bends and with the least number of vertices possible. As an application of our algorithm, we provide an upper bound on the number of bends of EPG models of VPT $\cap$ EPT graphs.
\end{abstract}

% keywords can be removed
\keywords{Grid embedding\and Number of bends\and EPG models}

\section{Introduction}\label{intro}
The problem of \emph{grid embedding of a graph} $G$ is that of drawing $G$ onto a rectangular two-dimensional grid (called simply \emph{grid}) such that each vertex $v \in V(G)$ corresponds to a grid point (an intersection of a horizontal and a vertical grid line) and the edges of $G$ correspond to  paths of the grid. Grid embedding of graphs has been considered with different perspectives \cite{aggarwal1985multi,aggarwal1991multilayer,beck2020puzzling,liu1998linear,schnyder1990embedding,tamassia1989planar}. In \cite{aggarwal1985multi}, the authors described an algorithm for embedding planar graphs and showed that the resulting embedding has edges with at most $6$ bends. The same authors later showed in \cite{aggarwal1991multilayer} a simpler approach which provides an embedding in which each edge has at most $4$ bends. Embeddings having edges with at most $4$ bends were also considered in \cite{tamassia1989planar}, where a linear-time algorithm is presented.
In \cite{liu1998linear}, linear-time algorithms were described for embedding planar graphs such that the resulting embedding is guaranteed to have edges with at most $2$ bends, with the exception of the octahedron, for which at most two edges with $3$ bends may be produced. In this same paper, the authors also provided an upper bound on the total number of bends of their embeddings. Note that, all of these results are related to the problem of finding embeddings of graphs in which the aim is to minimize the number of bends of the edges of the graph and/or the total number of bends.

In this paper, we introduce a similar problem: given a planar graph~$G$, find a grid embedding of $G$ in which the edges of $G$ correspond to pairwise non-intersecting  paths of the grid, each one having no bends, such that the maximum number of bends, over all paths of $G$, is minimized. So, this new problem shifts the focus on the number of bends of edges of $G$ to the number of bends of paths of $G$. This problem has not yet been considered in the literature so far. For instance, the tree in Figure~\ref{fig: M_1} is drawn in such a way that there is a path having $5$ bends (the path joining $o$ and $m$), and $5$ is the maximum number of bends in that drawing. However, this maximum number of bends can be decreased to $3$ (path joining $e$ and $f$), as Figure~\ref{fig: M_2} illustrates.
\begin{figure}[htb]

\center
\subfigure[][]{\includegraphics[scale=0.6]{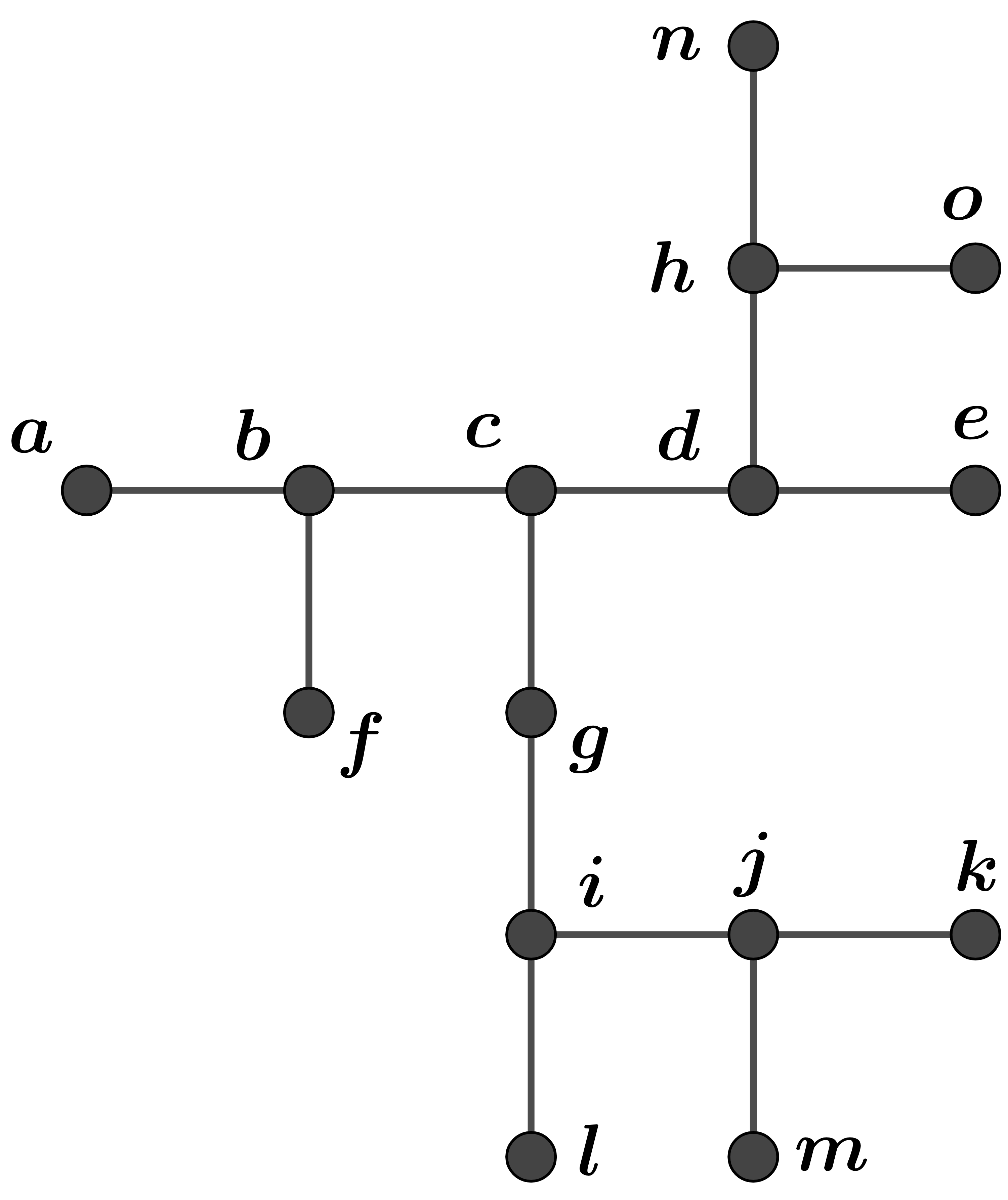} \label{fig: M_1}}
\qquad
\subfigure[][]{\includegraphics[scale=0.6]{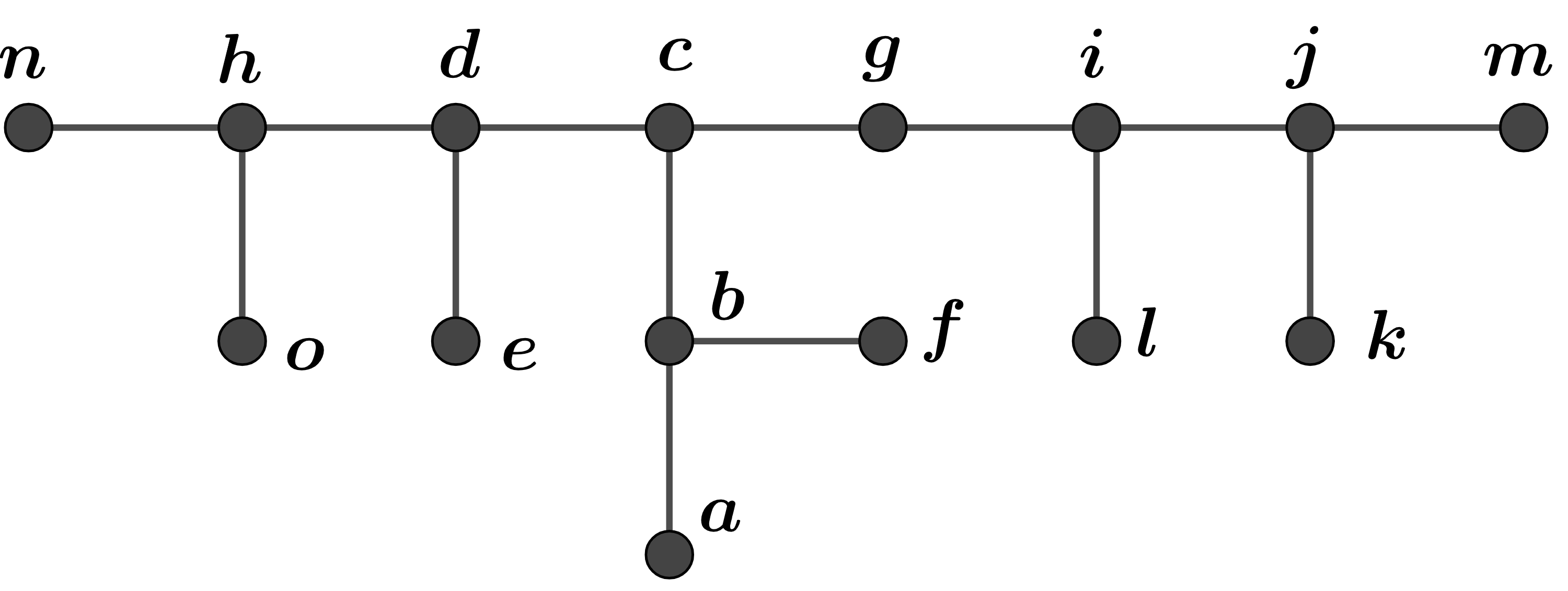} \label{fig: M_2}}
\caption{Two possible models $\mathcal{M}_1$ (left) and $\mathcal{M}_2$ (right) of the same tree $T$.}
\label{fig: Models of T}

\end{figure}

This problem will find an application obtaining certain grid models of VPT $\cap$ EPT graphs We will define these classes of graphs next, and the application will be deferred to Section~\ref{sec: EPG representations of VPT}. Given a tree $T$, called \emph{host tree}, and a set $\mathcal{P}$ of paths in $T$, the \emph{vertex (resp. edge) intersection graph of paths in a tree} (VPT (resp. EPT)) of $\mathcal{P}$ is the graph denoted by VPT($\mathcal{P}$) (resp. EPT($\mathcal{P}$)) having $\mathcal{P}$ as vertex set and two vertices are adjacent if the corresponding paths have in common at least one vertex (resp. edge). We say that $\langle T, \mathcal{P} \rangle$ is a \emph{VPT (resp. EPT) model} of $G$. Figure~\ref{fig: VPT_EPT: a} shows a host tree and the family of paths $\mathcal{P}=\{A, B, C, D\}$ in it, as well as the corresponding VPT($\mathcal{P}$) and EPT($\mathcal{P}$) graphs in Figures~\ref{fig: VPT_EPT: b} and~\ref{fig: VPT_EPT: c} respectively.
\begin{figure}[htb]

\center
\subfigure[][]{\includegraphics[scale=0.3]{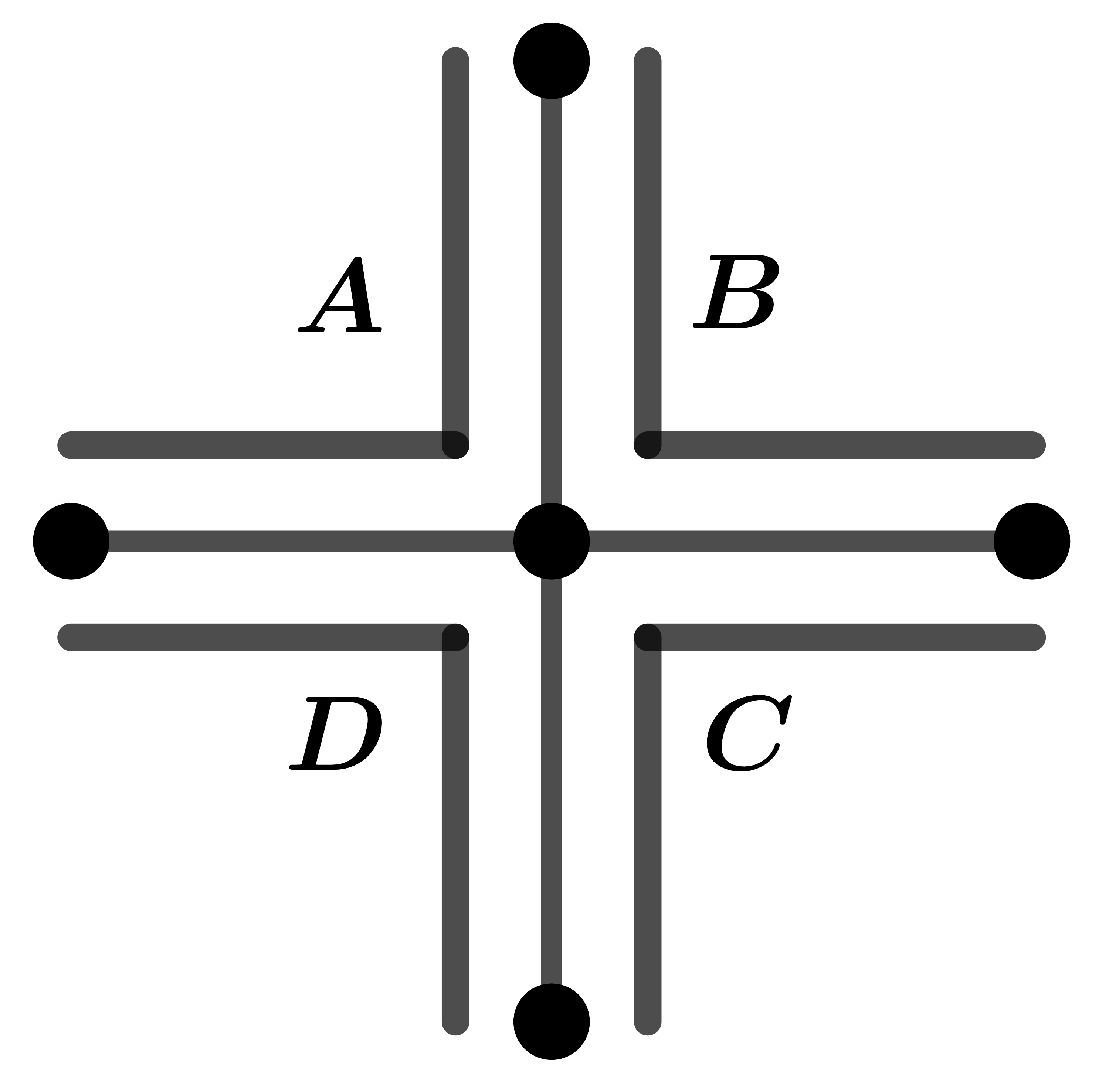} \label{fig: VPT_EPT: a}}
\hspace{15mm}
\subfigure[][]{\includegraphics[scale=0.3]{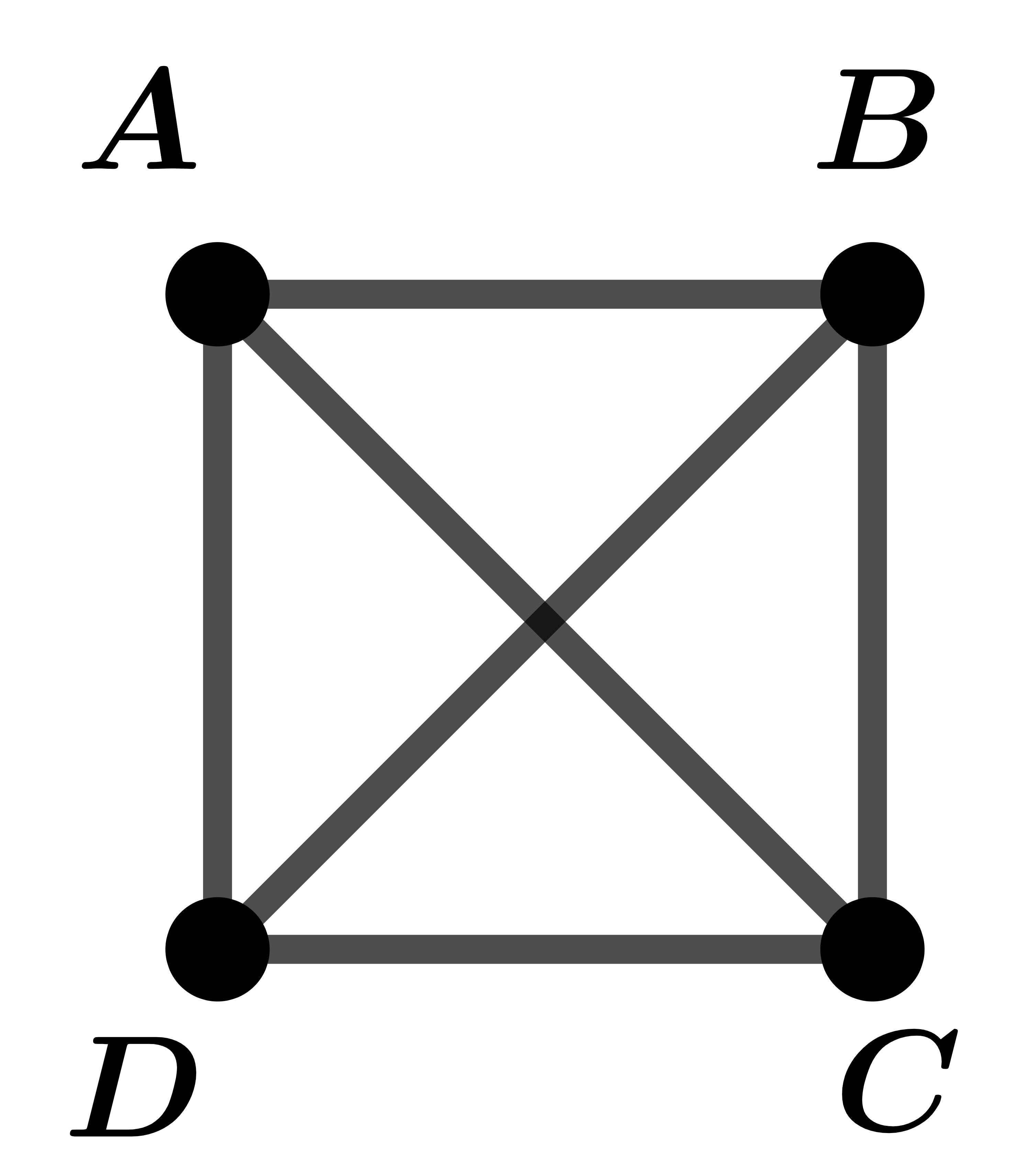} \label{fig: VPT_EPT: b}}
\hspace{15mm}
\subfigure[][]{\includegraphics[scale=0.3]{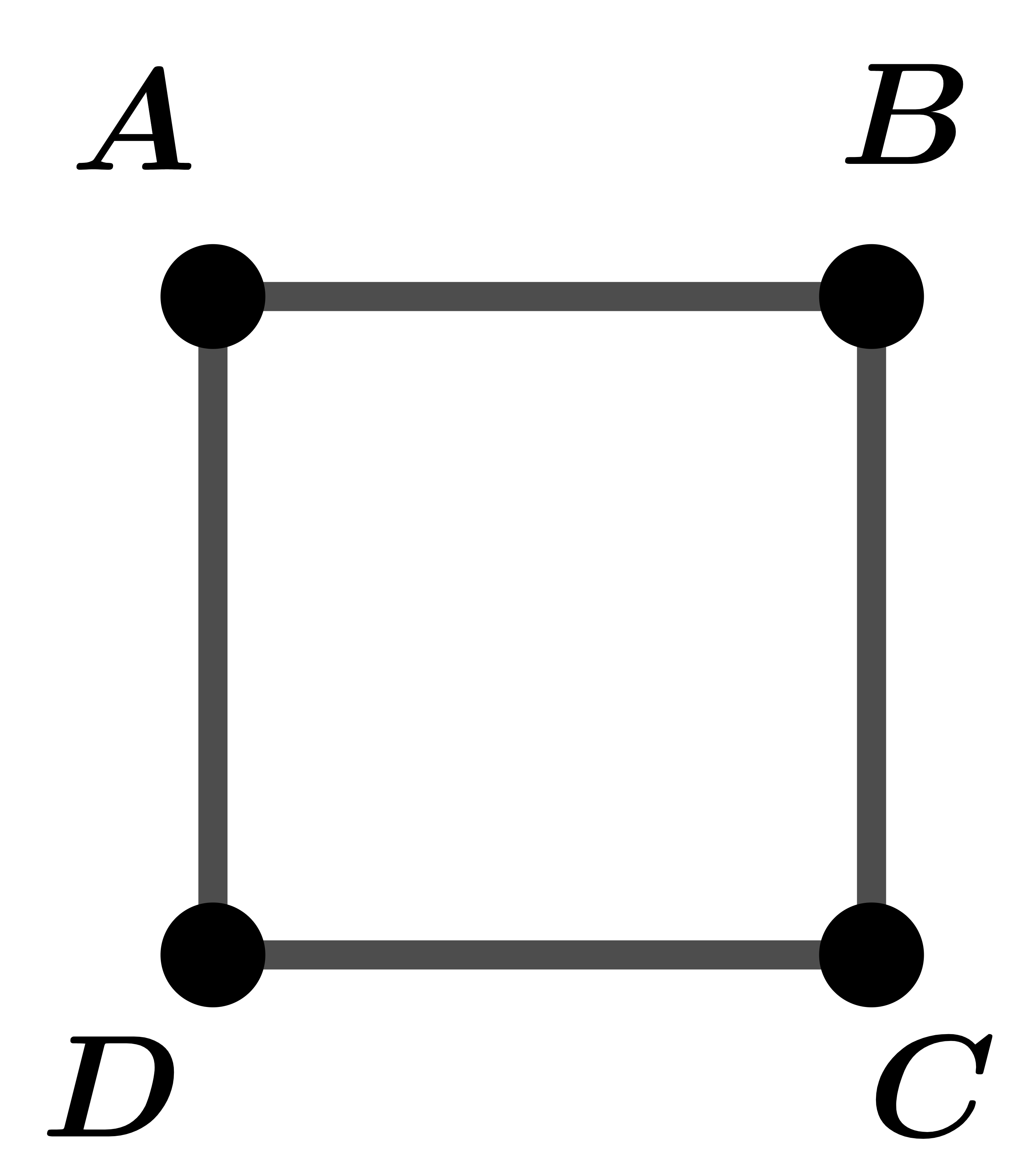} \label{fig: VPT_EPT: c}}
\hspace{15mm}
\subfigure[][]{\includegraphics[scale=0.3]{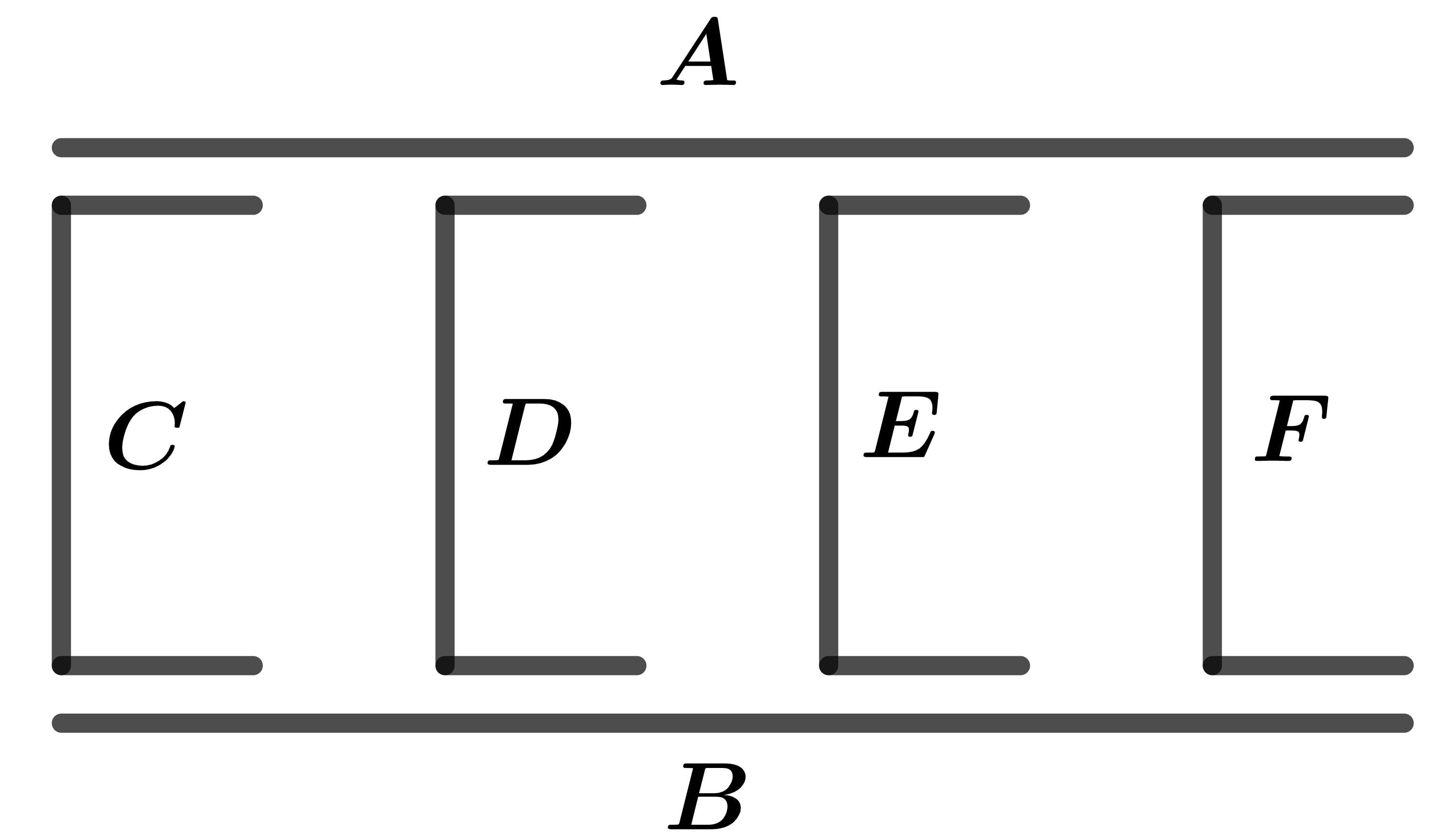} \label{fig: VPT_EPT: d}}
\hspace{15mm}
\subfigure[][]{\includegraphics[scale=0.3]{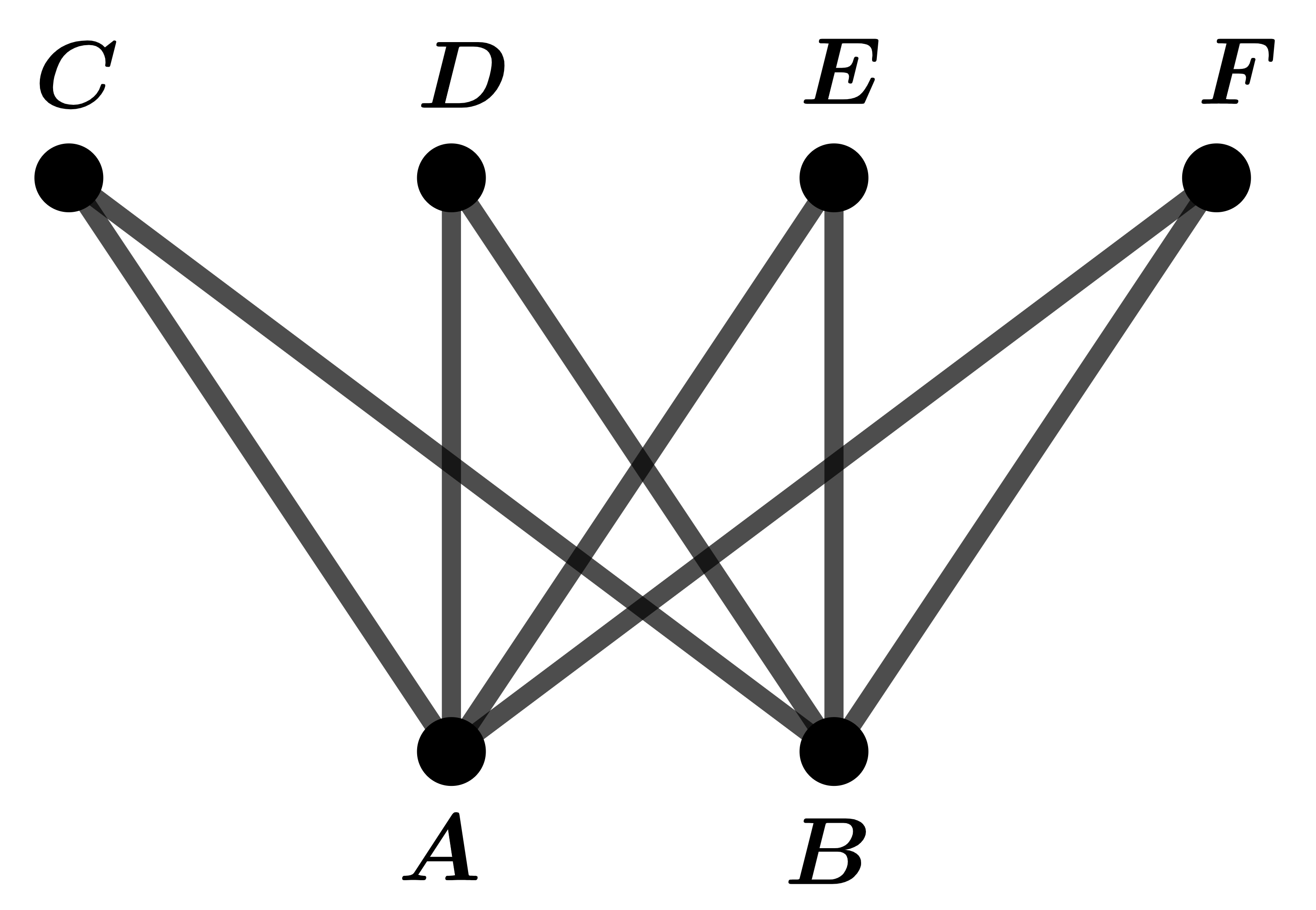} \label{fig: VPT_EPT: e}}
\caption{(a) Family of paths $\mathcal{P}=\{A, B, C, D\}$; (b) The corresponding VPT($\mathcal{P}$) and (c) EPT($\mathcal{P}$) graphs; (d) A B$_{2}$-EPG model $\mathcal{Q}$ and (e) its corresponding EPG graph.} 
\label{fig: VPT_EPT}

\end{figure}

Instead of host trees, the graphs we are interested have a grid $\mathcal{G}$ as the underlying structure from which a family of paths $\mathcal{P}$ is taken. A graph $G$ is an \emph{edge intersection graph of paths of a grid} if there is a collection of paths $\mathcal{P}$ in a grid $\mathcal{G}$ such that each vertex in $G$ corresponds to a member of $\mathcal{P}$ and two vertices are adjacent if, and only if, the corresponding paths have in common at least one edge of the grid. We say that $\mathcal{P}$ is an \emph{EPG model} of $G$. EPG graphs were first introduced by Golumbic, Lipshteyn and Stern in 2009 (c.f. \cite{chung201950}) motivated from circuit layout problems. The same authors showed that every graph is an EPG graph. Figure~\ref{fig: VPT_EPT: e} illustrates the EPG graph corresponding to the family of paths presented in Figure~\ref{fig: VPT_EPT: d}.

A turn of a path at a grid point is called a \emph{bend} and the grid point in which a bend occurs is called a \emph{bend point}. An EPG model is a \emph{B$_{k}$-EPG model} if each path has at most $k$ bends. A graph that admits a B$_{k}$-EPG model is called B$_{k}$-EPG. Therefore, the graph $G$ defined in Figure~\ref{fig: VPT_EPT: e} is B$_{2}$-EPG, as the model shows. However, it is possible to show that there is a B$_{1}$-EPG model of $G$ and, thus, $G$ is also B$_{1}$-EPG. The problem of finding the least $k$ such that a given graph is B$_{k}$-EPG has been a question of great interest since EPG graphs were introduced. In particular, the B$_{1}$-EPG graphs have been widely studied. On the other hand, there are fewer results for graphs that are B$_{k}$-EPG, for $k\geq 2$. It was proved that the recognition problem for B$_1$-EPG is NP-complete and it was also proved that the recognition problem for B$_2$-EPG is also NP-complete. It is the standing conjecture that determining the least $k$ such that an arbitrary graph is B$_{k}$-EPG is NP-complete for general $k$. References to all mentioned results can be found in~\cite{chung201950}.

In Section~\ref{sec:embedding trees}, we discuss the concept of straight model of trees. It is shown an algorithm that determines one such model in which the maximum number of bends over all paths of a tree is minimized. In Section~\ref{sec: EPG representations of VPT}, we employ our algorithm to obtain straight models of host trees that are used to build EPG models of VPT $\cap$ EPT graphs. An upper bound on the number of bends of EPG models of VPT $\cap$ EPT graphs is then derived.  Concluding remarks are presented in Section~\ref{sec: Conclusion}.

\section{Embedding trees in a grid} \label{sec:embedding trees}

Let $T$ be a tree such that $\Delta (T) \leq 4$. Consider the problem of embedding such a tree in a grid $\mathcal{G}$, such that the vertices must be placed at grid points and the edges drawn as non-intersecting paths of $\mathcal{G}$ with no bends, which we will call a \emph{straight model} of $T$, or simply an \emph{s-model} of $T$. Figure~\ref{fig: Models of T} depicts two possible s-models corresponding to a same tree. Given a path $Q = v_1, v_2,\ldots, v_k$ of $T$, and an s-model $\mathcal{M}$ of $T$, the number of bends of $Q$ is defined to be the number of bends of the grid path consisting of the concatenation of the (straight) paths of the grid corresponding to the edges $(v_1, v_2), (v_2, v_3), \ldots, (v_{k-1}, v_k)$ in $\mathcal{M}$. In Figure~\ref{fig: M_1}, the path $o, h, d, c, g, i, j, m$ has $5$ bends, whereas the same path in Figure~\ref{fig: M_2} has only $1$ bend.

Among all possible s-models, consider the problem of finding one in which the maximum number of bends of a path of $T$, over all of them, is minimum. Note that, since every path of a tree is contained in a leaf-to-leaf path, among the paths that bend the most in a given s-model is a leaf-to-leaf path, and therefore those are the only ones to be considered. More formally, let $\mathcal{M}(T)$ be the set of all possible s-models of a given tree $T$ and $u, v\in V(T)$ be leaves of $T$. The number of bends of the path connecting $u$ and $v$ in an s-model $\mathcal{M}\in\mathcal{M}(T)$ is denoted by $b_{\mathcal{M}}(u, v)$. Therefore, $b_{\mathcal{M}_1}(o,m) = 5$ and $b_{\mathcal{M}_2}(o,m) = 1$. Let
$$b(\mathcal{M})=\max\{b_{\mathcal{M}}(u, v)\mid\textrm{$u$ and $v$ are leaves in $T$}\}$$ denote the number of bends of the path that bends the most in $\mathcal{M}$, and
$$b(T)=\min\{b(\mathcal{M})\mid \mathcal{M}\in\mathcal{M}(T)\}$$
the minimum number of bends of an s-model, over all of them. Figure~\ref{fig: M_1} depicts an s-model~$\mathcal{M}_1$ of a tree $T$ such that $b(\mathcal{M}_1) = 5$, and therefore $b(T) \leq 5$. Figure~\ref{fig: M_2} shows another s-model $\mathcal{M}_2$ of $T$ for which $b(\mathcal{M}_2) = 3$ and, therefore, $b(T) \leq 3$. It is possible to show that no s-model $\mathcal{M}$ of $T$ has $b(\mathcal{M}) = 2$ and, thus, $b(T) = 3$.

We provide an algorithm that determines $b(T)$ and finds an s-model $\mathcal{M}$ for which $b(\mathcal{M}) = b(T)$. Before describing the algorithm, some more terminology will be required.

Given $\mathcal{M} \in \mathcal{M}(T)$, let $b^{\ell}_{\mathcal{M}}(p, v)$ denote the maximum number of bends found in a single path having as extreme vertices $p$ and a leaf of $T$, over all paths that contain $v \in V(T)$. That is,
$$b^{\ell}_{\mathcal{M}} (p, v)= \max\{b_{\mathcal{M}}(p, f) \mid \text{$f$ is a leaf of $T$ and the path connecting $p$ and $f$ contains $v$}\}\, .$$
Also, define
$$b^{\ell}_{T} (p, v)= \min\{b^{\ell}_{\mathcal{M}}(p, v) \mid \mathcal{M} \in \mathcal{M}(T)\}\, .$$
As examples, in Figure~\ref{fig: Models of T}, $b^{\ell}_{\mathcal{M}_2} (i, l)= 0 = b^{\ell}_{T} (i, l)$, $b^{\ell}_{\mathcal{M}_2} (i, j)=1 = b^{\ell}_{T} (i, j)$, $b^{\ell}_{\mathcal{M}_2} (i, g)~=~2~=~b^{\ell}_{T} (i, g)$. Note that $b^{\ell}_{\mathcal{M}_1} (b, c)=3$, whereas $b^{\ell}_{\mathcal{M}_2} (b, c)=2$. 

Let $\mathcal{M} \in \mathcal{M}(T)$ and $v \in V(T)$. Let $\{u^i_{\M}(v) \mid 1 \leq i \leq d(v) \}$ be $N(v)$ and $b^i_{\M}(v) = b^{\ell}_{\M} (v, u^i_{\M}(v))$. For $d(v) < i \leq 4$, define ``virtual'' neighbors $u^i_{\M}(v) = \emptyset$ for which $b^i_{\M}(v) = -1$. Assume that the neighbors (both real and virtual) are ordered so that $b^i_{\M}(v) \geq b^{i+1}_{\M}(v)$ for all $1 \leq i < 4$. As examples, $u^1_{\M_2}(i) = g$ (and $b^1_{\M_2}(i) = 2$), $u^2_{\M_2}(i) = j$ (and $b^2_{\M_2}(i) = 1$), $u^3_{\M_2}(i) = l$ (and $b^3_{\M_2}(i) = 0$), and $u^4_{\M_2}(i) = \emptyset$ (and $b^4_{\M_2}(i) = -1$). Let  $(p,v) \in E(T)$. Let $\{u^i_{\M}(p,v) \mid 1 \leq i < d(v) \}$ be $N(v) \setminus \{p\}$ and $b^i_{\M}(p,v) = b^{\ell}_{\M}(v, u^i_{\M}(p,v))$. For $d(v) \leq i < 4$, also define $u^i_{\M}(p,v) = \emptyset$ for which $b^i_{\M}(p,v) = -1$. Again, assume that the neighbors $u^i_{\M}(p,v)$ are ordered according to their respective values of $b^i_{\M}(p,v)$. As examples,  $u^1_{\M_2}(g,i) = j$ (and $b^{1}_{\M_2}(g,i) = 1$), $u^2_{\M_2}(g,i) = l$ (and $b^{2}_{\M_2}(g,i) = 0$), and $u^3_{\M_2}(g,i) = \emptyset$ (and $b^3_{\M_2}(g,i) = -1$). Besides, $u^1_{\M_2}(i,g) = c$ (and $b^1_{\M_2}(i,g) = 2$), and $u^2_{\M_2}(i,g) = u^3_{\M_2}(i,g) = \emptyset$ (and $b^2_{\M_2}(i,g) = b^3_{\M_2}(i,g) = -1$). The s-model may be omitted from these notations when it is unambiguous.

%Call $v$ \emph{critical} if $b_2 = b_3$ and, in this case, call $u_2, u_3$ of \emph{critical pair}. In Figure~\ref{fig: M_2}, $i$ is not critical, whereas $c$ is with critical pair $b,d$.    

A tree $T$ can be built from a single vertex $v_0$ by a sequence $v_1,v_2,\ldots,v_{n-1}$ of vertex additions, each new vertex $v_i$ adjacent to exactly one vertex $p_i$ of $T$ $[\{v_0,v_1,\ldots,v_{i-1}\}]$ for all $1~\leq~i~<~n$. We will call that $T$ \emph{is incrementally built by}  $(v_0, \emptyset), (v_1, p_1), \ldots, (v_{n-1},p_{n-1})$. For instance, the tree $T$ of Figure~\ref{fig: Models of T} is incrementally built by $(g, \emptyset), (i, g), (l, i), (c, g), (j, i),$ $(k, j), (m, j), (d, c), (b, c), (f, b), (a, b), (e, d), (h, d), (n, h), (o, h)$. 

Let $v \in V(T)$ and $\M \in \M(T)$. We say that $v$ is \emph{balanced} if $u^1(v)$ and $u^2(v)$ are mutually in the same horizontal or vertical grid line in $\M$ (and, therefore, so are $u^3(v)$ and $u^4(v)$).

We are ready to present  Algorithm~\ref{alg}, which consists  of adding iteratively vertices to $T$ and, for each new vertex $v$, traversing $T$ in post-order taking $v$ as the root. The operation to be carried out in each visited vertex is to balance $v$ if it is not balanced. To illustrate its execution,  Figure~\ref{fig:alg_example} presents the output of the algorithm having the tree $T$ of Figure~\ref{fig: Models of T} as input. Each s-model corresponds to a partial s-model of the tree, as it is at the end of each iteration. Regarding the time complexity of the algorithm, it is possible to keep the values of $u^i_{\M}(v)$ stored for each $v \in V(T)$ and $1 \leq i \leq 4$, and update them in constant time right after the balance step, based on which subtrees had their positions exchanged and on the respective values of $u^i_{\M}(w)$ from the neighbors $w \in N(v)$. Since the algorithm performs $n-1$ post-order traversals in $T$, the algorithm runs in $O(n^2)$ time. The remaining of the section is devoted to the correctness of the algorithm, that is, to prove that if $\M$ is the s-model produced by Algorithm~\ref{alg} on input $T$, then $b(\M) = b(T)$. We also provide some numerical results concerning the number of bends of trees.
\begin{algorithm}

\begin{algorithmic}[0]
\Require a tree $T$ such that $\Delta(T) \leq 4$.
\Ensure an s-model $\mathcal{M}$ of $T$ such that $b(\mathcal{M}) = b(T)$.
\State Let $S = (v_0, \emptyset), (v_1, p_1), \ldots, (v_{n-1},p_{n-1})$ be such that $T$ is incrementally built by $S$.
\State Let $\mathcal{M}$ be an s-model having a single grid point representing $v_0$.
 \For {$i$ $\leftarrow$ $1$ \textbf{to} $n-1$}
    \State Add to $\mathcal{M}$ the vertex $v_i$ attached to the grid point of $p_i$, in any free horizontal or vertical grid line of $p_i$
    \State\textsc{Balance}($\mathcal{M}$, $v_i$, $p_i$) 
 \EndFor\\
    
\Procedure{Balance}{$\mathcal{M}$, $p$, $v$}
    \For {$u \in N(v) \setminus \{p\}$}
        \State\textsc{Balance}($\mathcal{M}$, $v$, $u$) 
    \EndFor
    \State{\parbox{\textwidth}{If $v$ is not balanced, then make it balanced by rearranging in $\mathcal{M}$ the drawing of the four subtrees of $v$\\ rooted at $u^i(v)$ (for $1 \leq i \leq 4$), potentially rotating and rescaling them to fit [\textbf{\textit{balance step}}]}}
\EndProcedure
\end{algorithmic}
\caption{Determining $b(T)$}\label{alg}
\end{algorithm}
\begin{comment}

\begin{alg}
%\SetAlgoLined
%\SetAlgoNoEnd
%\DontPrintSemicolon
%\SetKwInOut{Input}{Input}
%\SetKwInOut{Output}{Output}

\SetKwFunction{FBalance}{\textsc{Balance}}
\SetKwFunction{FAdjust}{\textsc{Adjust}}
\SetKwProg{Fn}{Procedure}{:}{}
 \Input{a tree $T$ such that $\Delta(T) \leq 4$}\\
 \Output{a model $\mathcal{M}$ of $T$ such that $b(\mathcal{M}) = b(T)$}\\
 Let $S = (v_0, \emptyset), (v_1, p_1), \ldots, (v_{n-1},p_{n-1})$ be such that $T$ is incrementally built by $S$\\
 Let $\mathcal{M}$ be a model having a single vertex $v_0$ at some grid point\\
 \For{$i$ $\leftarrow$ $1$ \emph{\KwTo} $n-1$}{
    Add to $\M$ the vertex $v_i$ attached to the grid point of $p_i$, in any free horizontal or vertical grid line of $p_i$\\
    \FBalance{$\M$, $v_i$, $p_i$}
 }
 
\Fn{\FBalance{$\M$, $p$, $v$}}{
    \For{$u \in N(v) \setminus \{p\}$}{
       \FBalance{$\M$, $v$, $u$}
    }
    If $v$ is not balanced, then make it balanced by rearranging in $\M$ the drawing of the four subtrees of $v$ rooted at $u^i(v)$ (for $1 \leq i \leq 4$), potentially rotating and rescaling them  to fit [\textbf{\textit{balance step}}]
}

 \label{alg}
\end{alg} 

\end{comment}
\begin{figure}[htb]
    \centering
    \subfigure[][]{\includegraphics[scale=0.6]{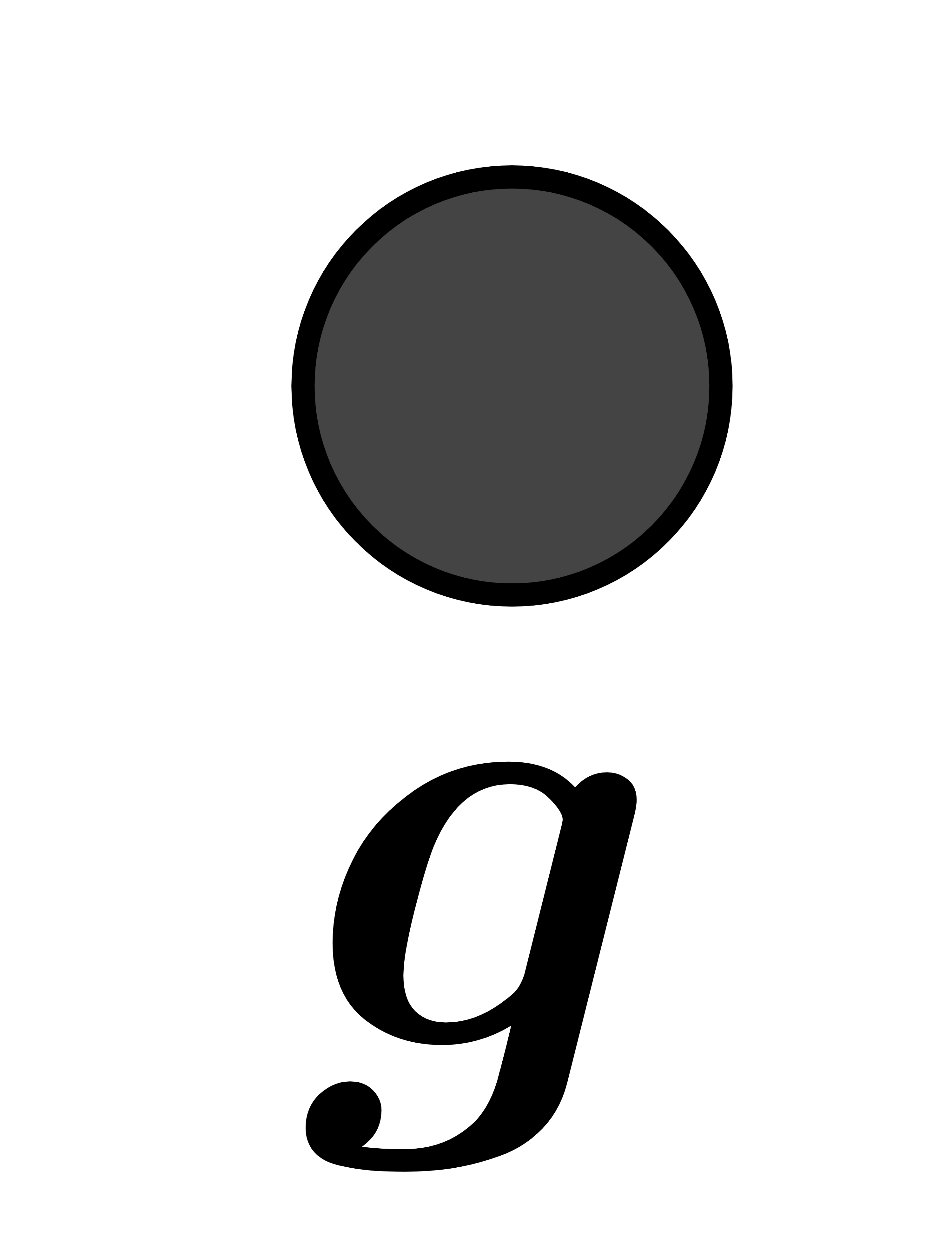}}
    \quad
    $\xrightarrow{i: 1-5}$
    \quad
    \subfigure[][]{\includegraphics[scale=0.6]{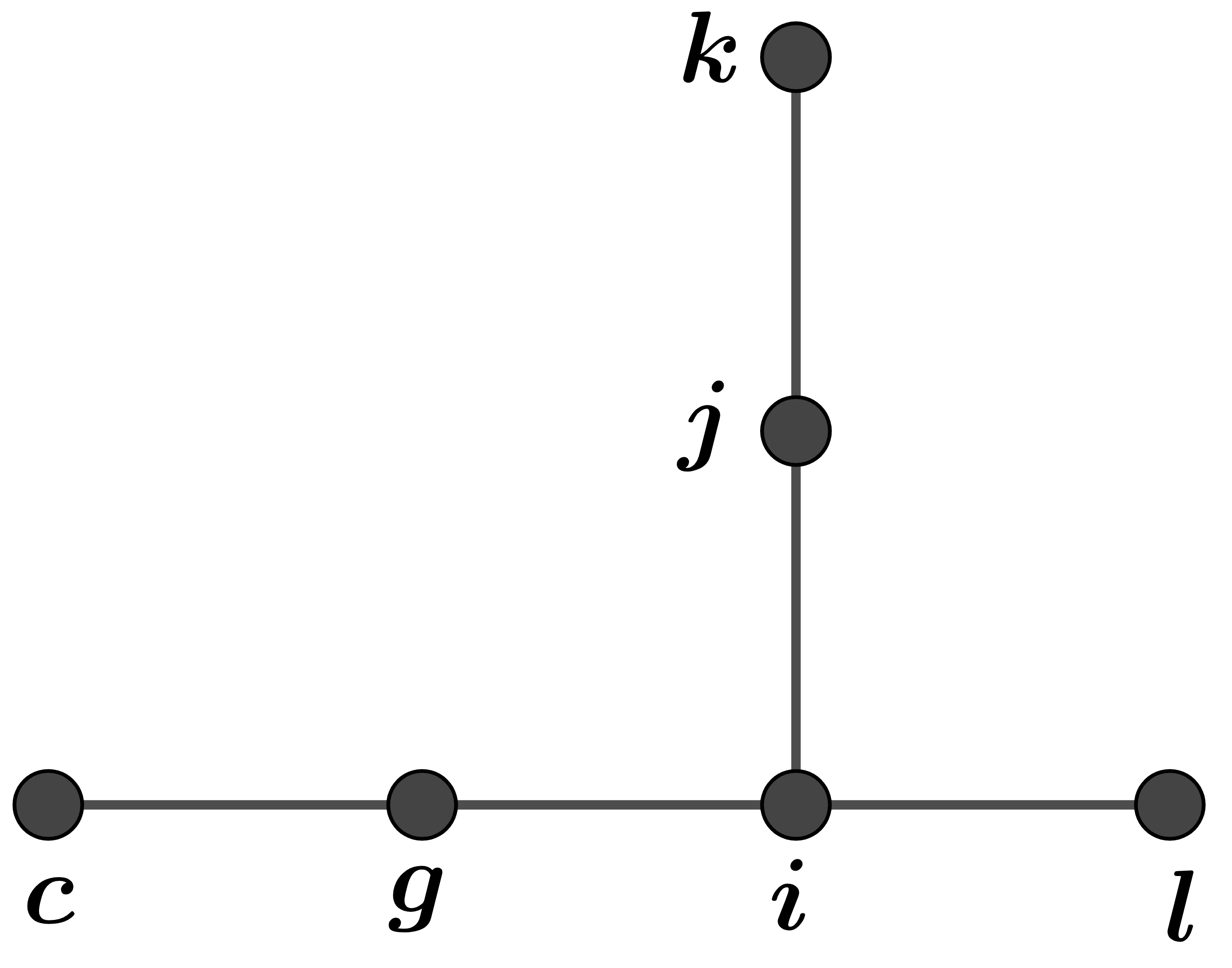}}
    \quad
    $\xrightarrow{}$
    \quad
    \subfigure[][]{\includegraphics[scale=0.6]{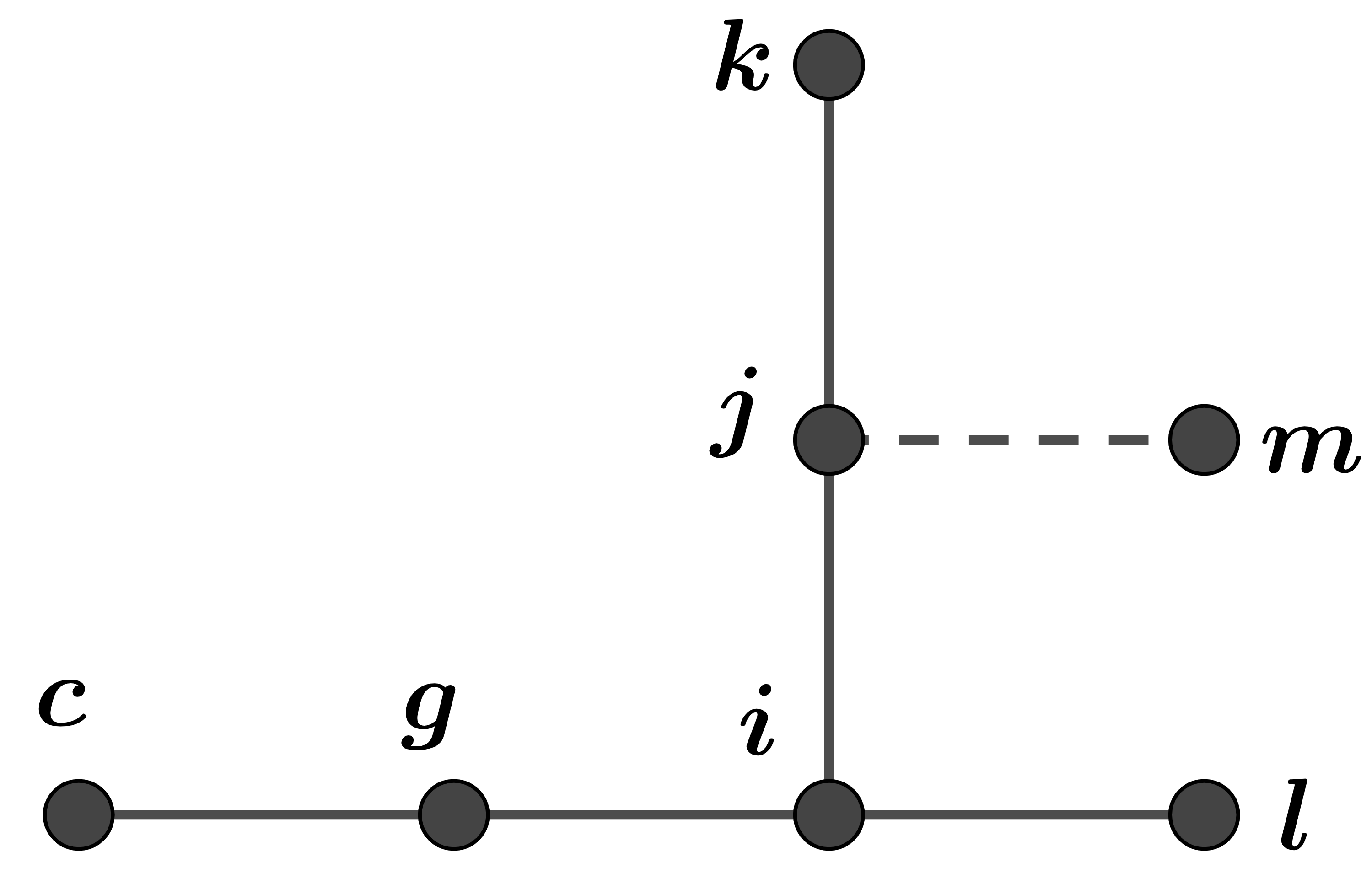}}
    \quad
    $\xrightarrow{\textsc{Balance}(\M, j, i)}$
    \quad
    \subfigure[][]{\includegraphics[scale=0.6]{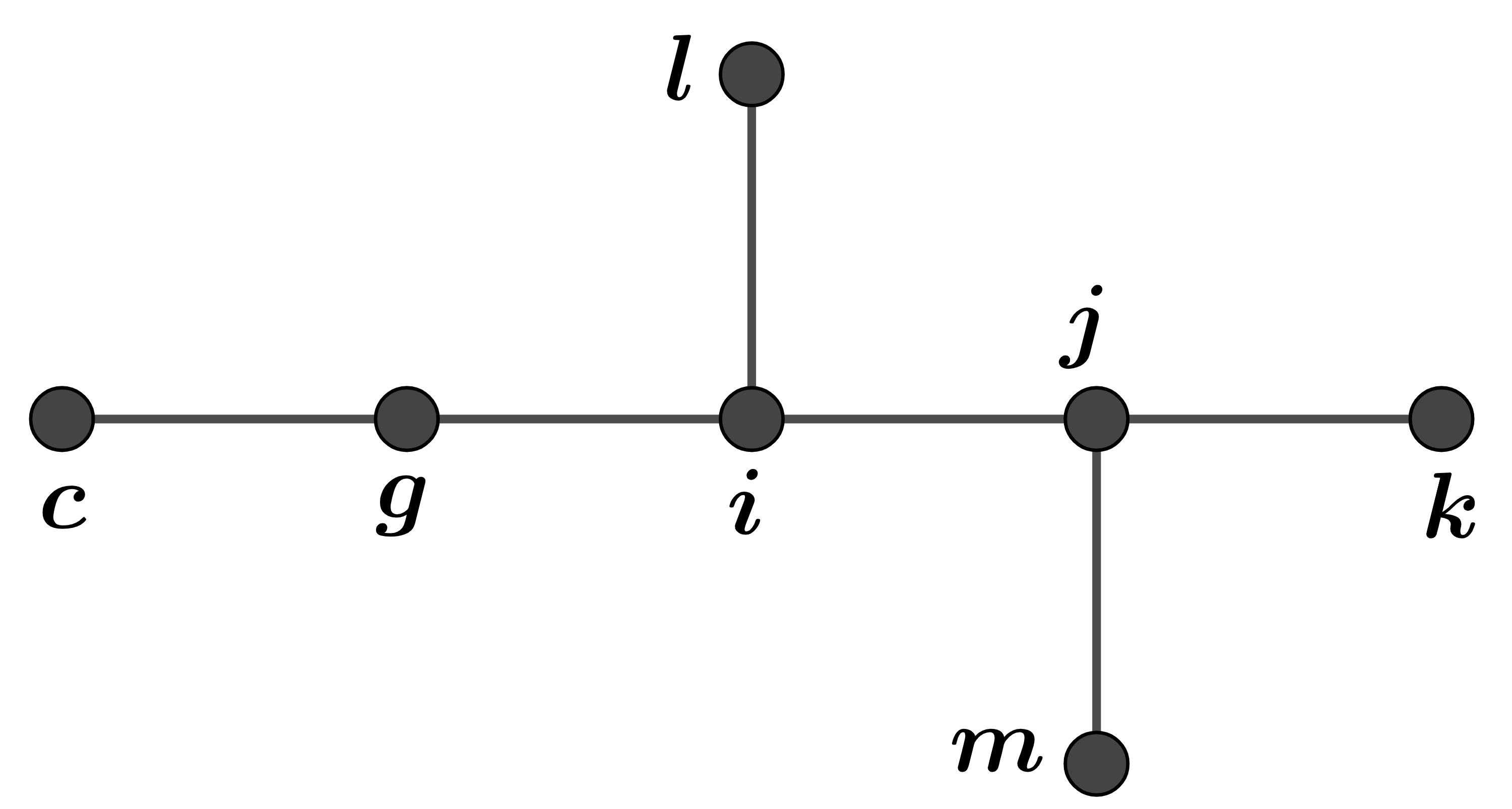}}
    \quad
    $\xrightarrow{i: 7-9}$
    \quad
    \subfigure[][]{\includegraphics[scale=0.6]{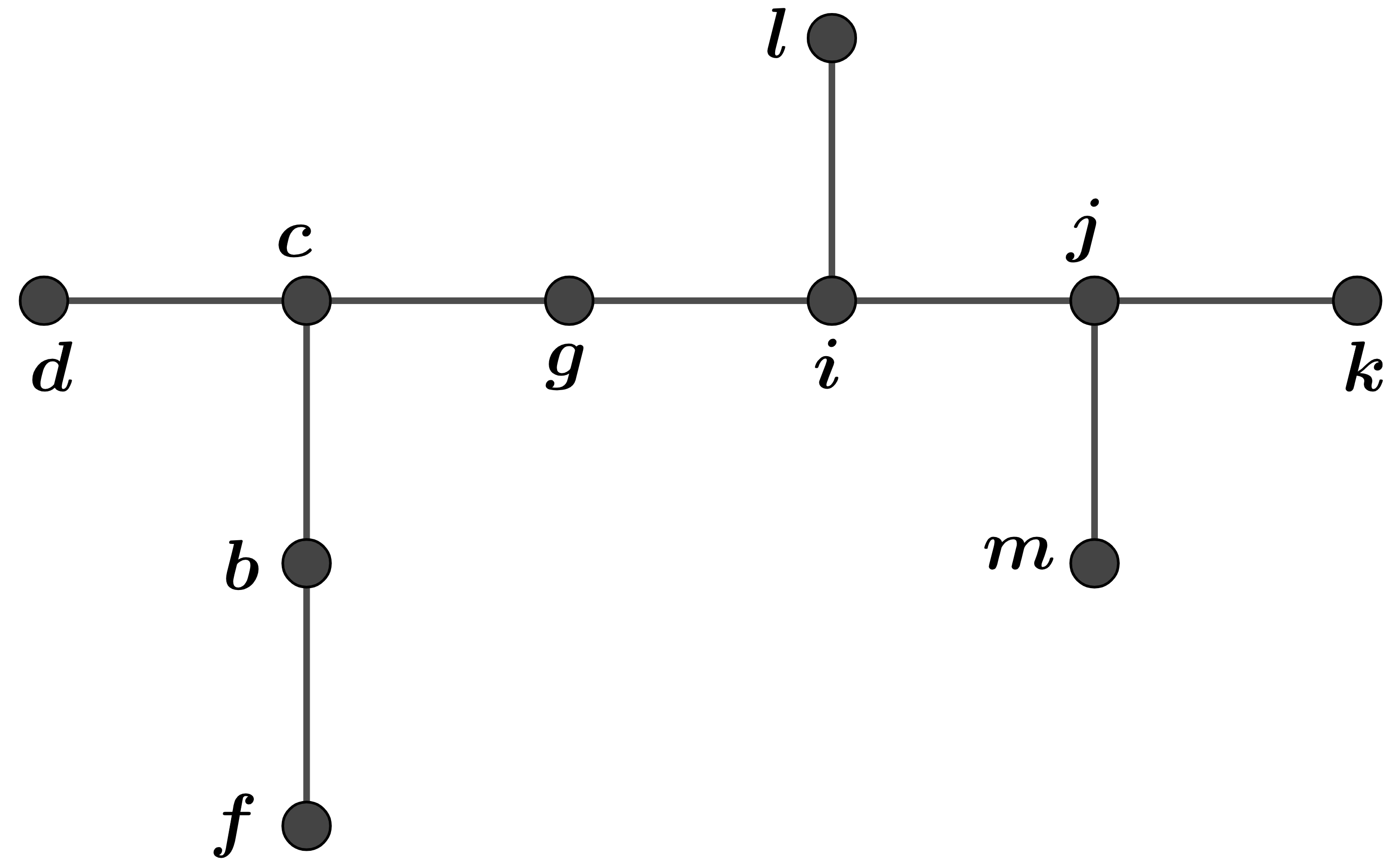}}
    \quad
    $\xrightarrow{}$
    \quad
    \subfigure[][]{\includegraphics[scale=0.6]{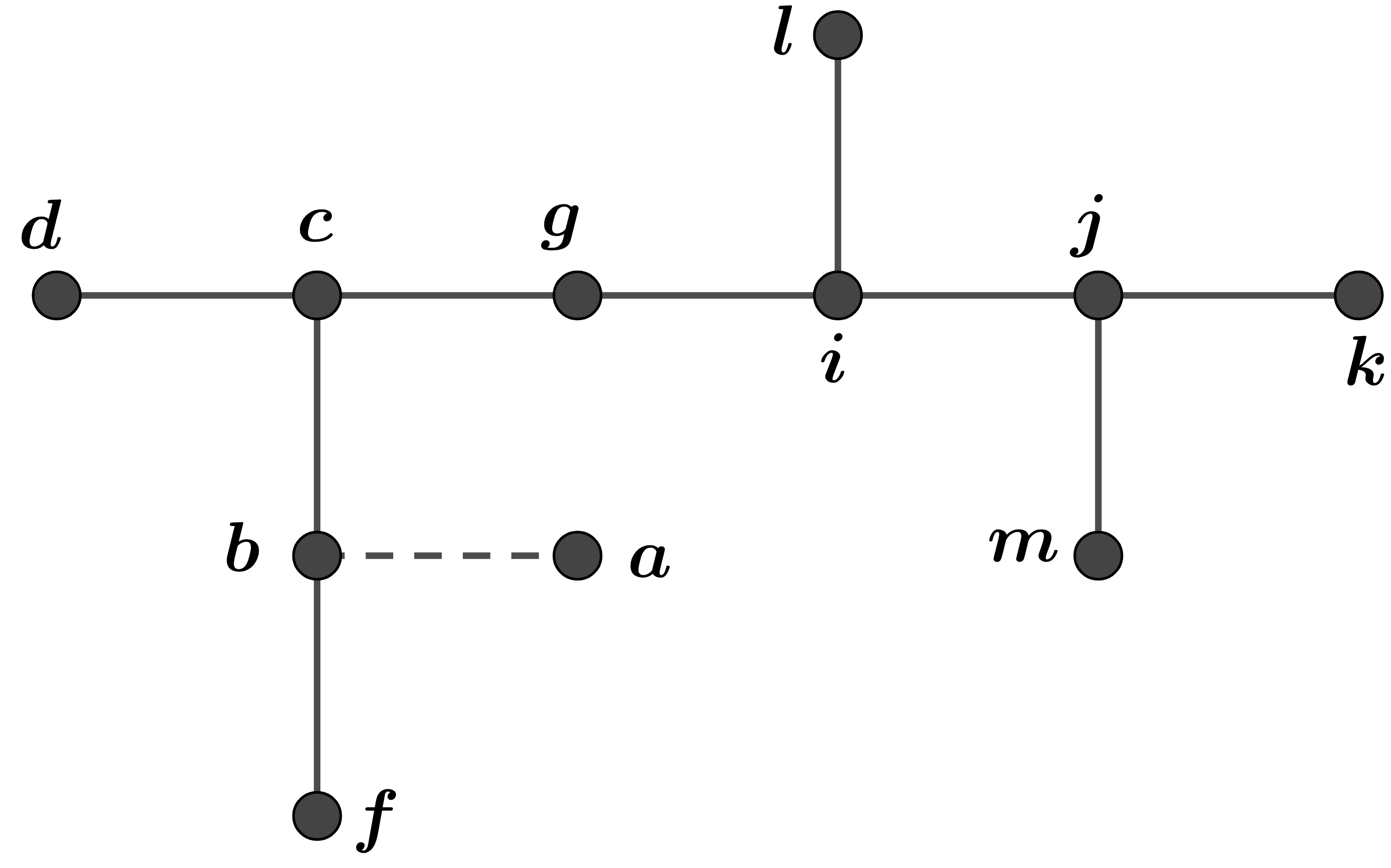}}
    \quad
    $\xrightarrow{\textsc{Balance}(\M, b, c)}$
    \quad
    \subfigure[][]{\includegraphics[scale=0.6]{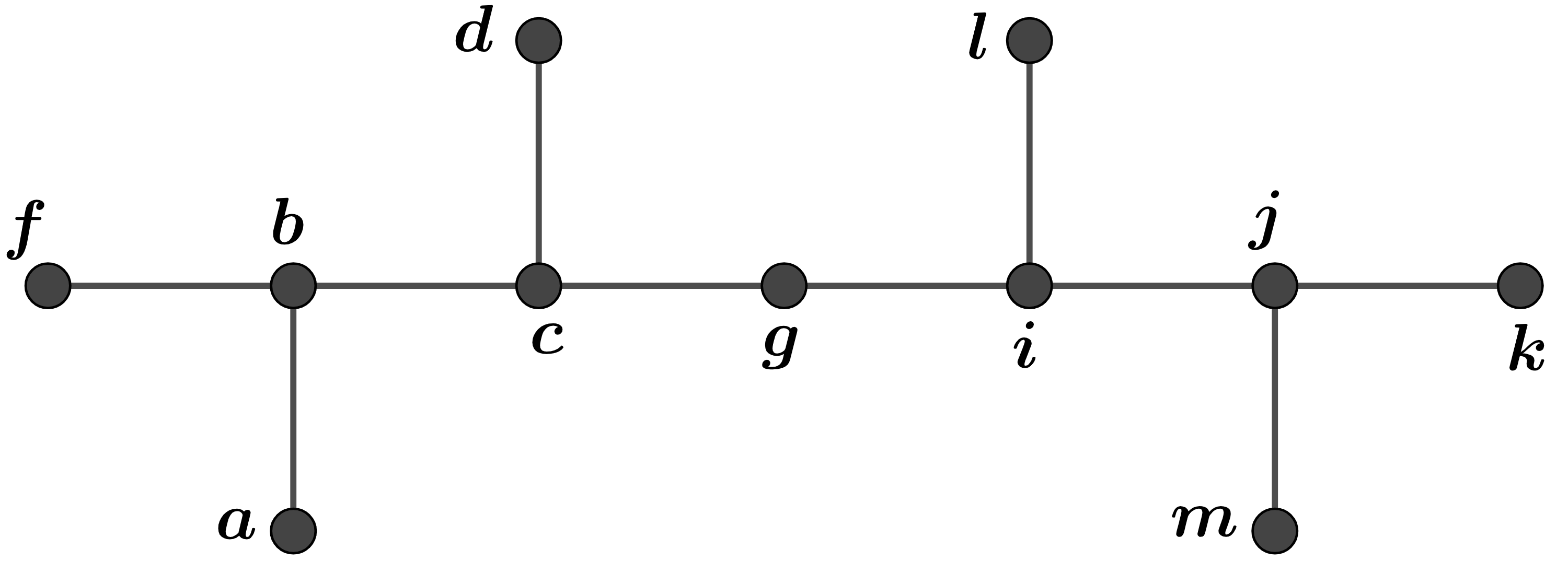}}
    \quad
    $\xrightarrow{i: 11-13}$
    \quad
    \subfigure[][]{\includegraphics[scale=0.6]{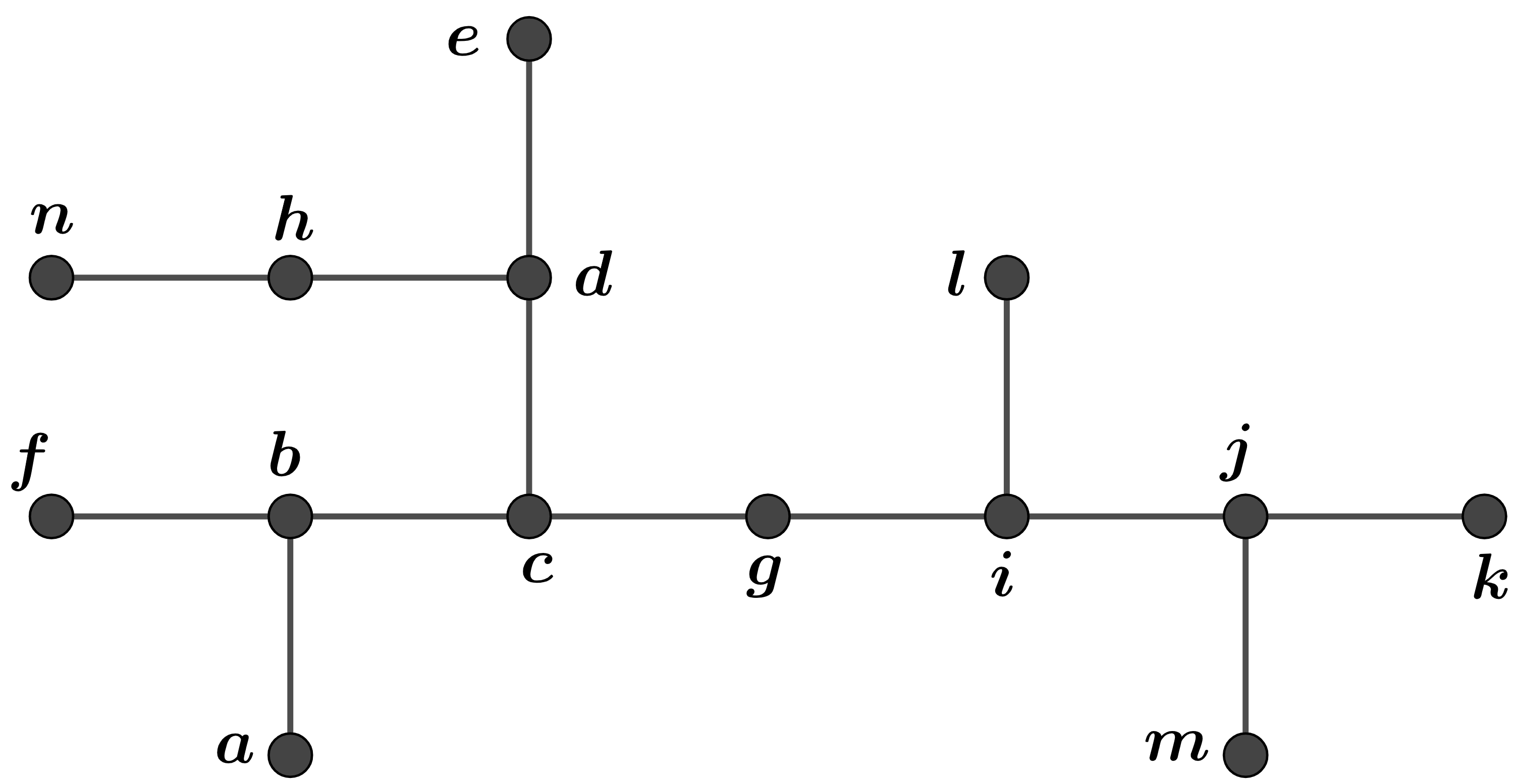}}
    \quad
    $\xrightarrow{}$
    \quad
    \subfigure[][]{\includegraphics[scale=0.6]{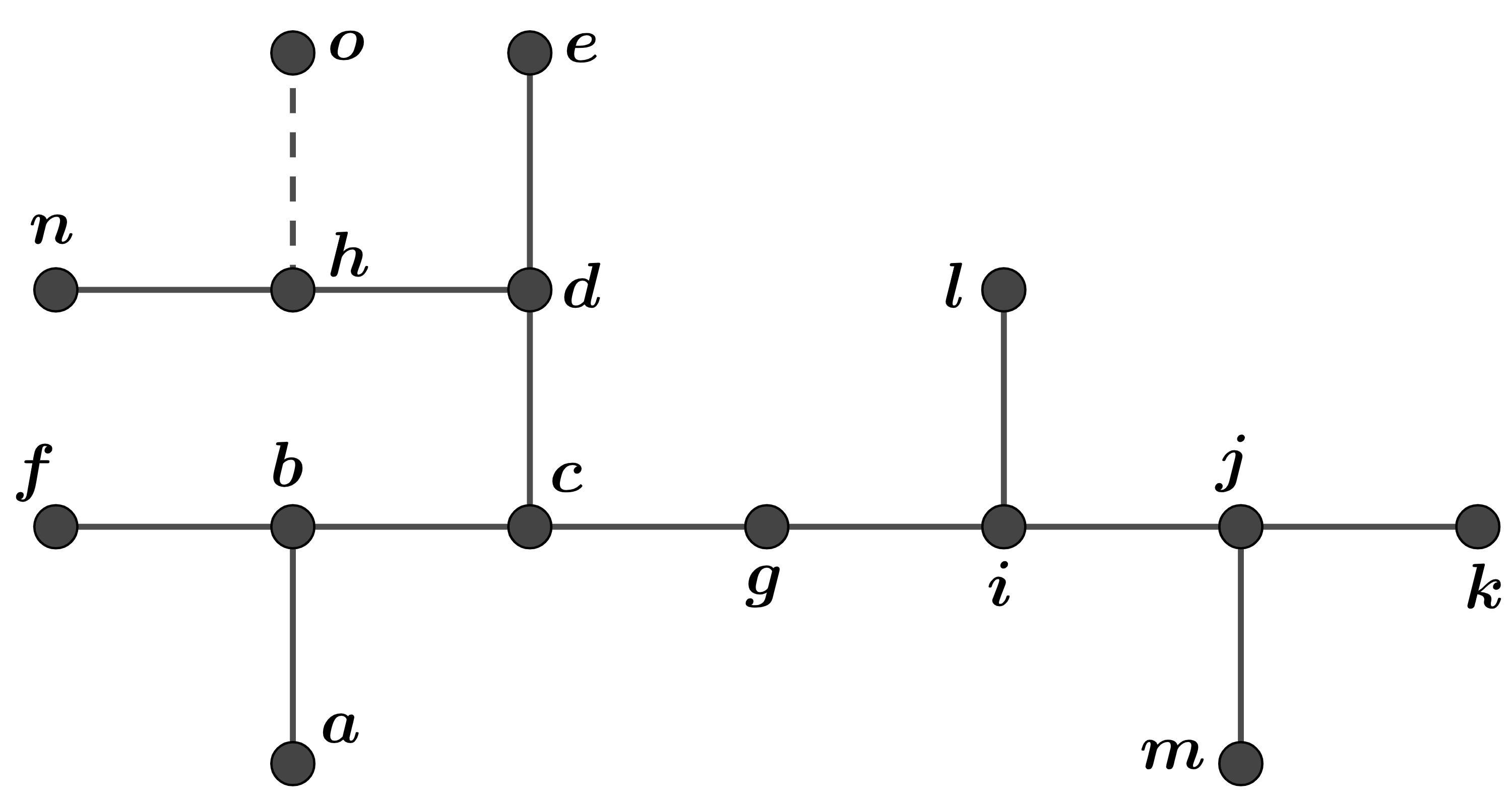}}
    \quad
    $\xrightarrow{\textsc{Balance}(\M, d, c)}$
    \quad
    \subfigure[][]{\includegraphics[scale=0.6]{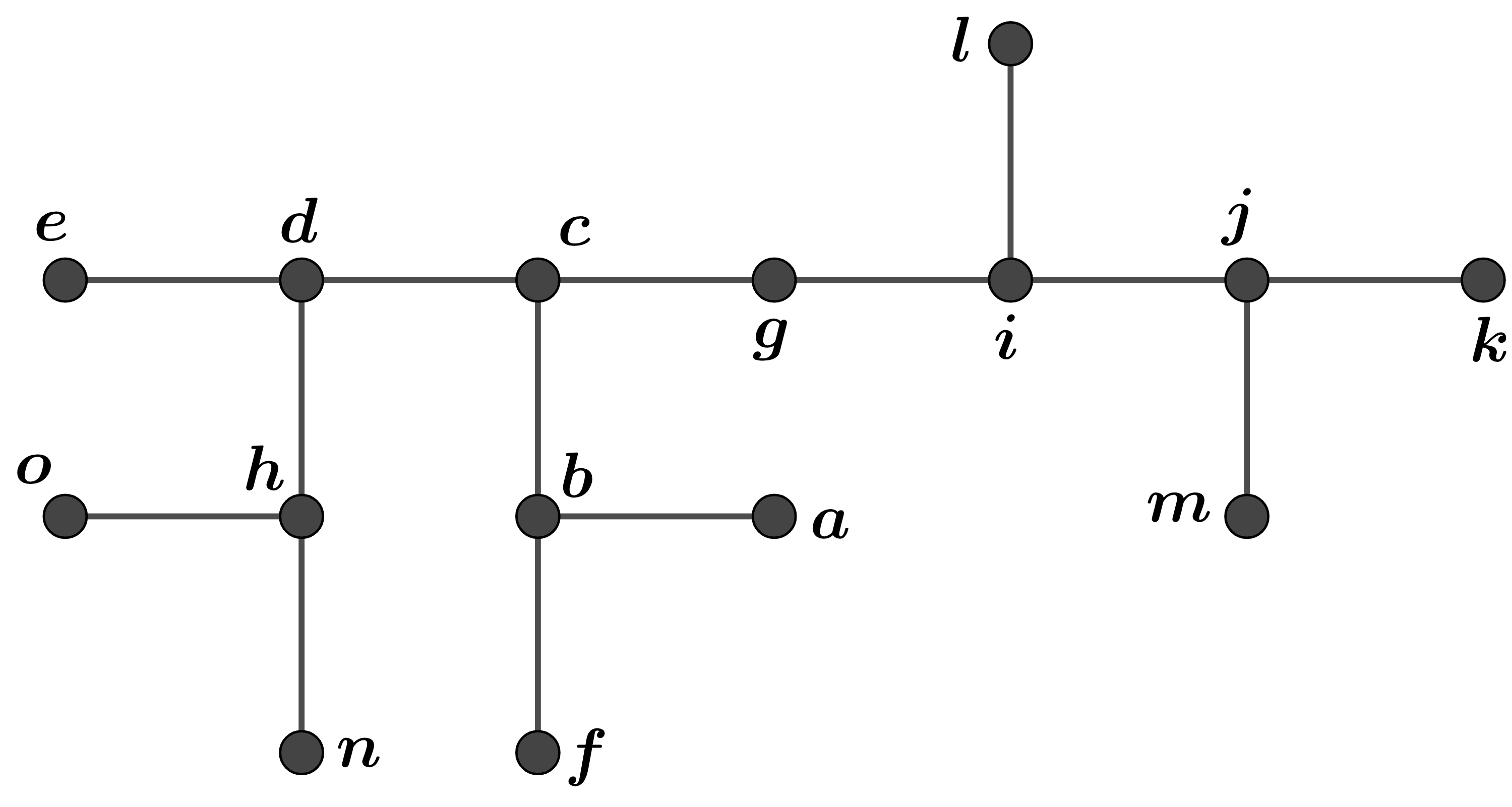}}
    \quad
    $\xrightarrow{\textsc{Balance}(\M, h, d)}$
    \quad
    \subfigure[][]{\includegraphics[scale=0.6]{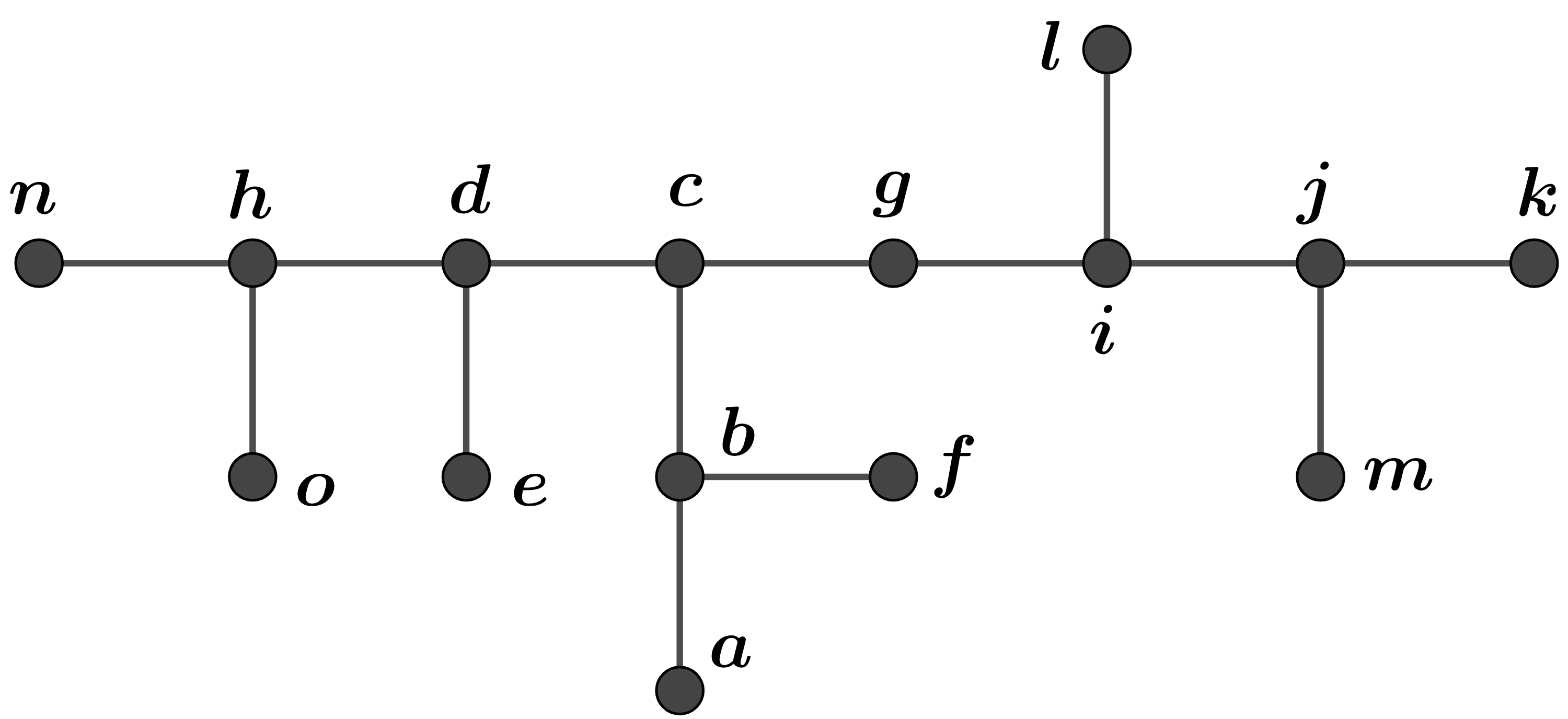}}
    \caption{Execution of the Algorithm~\ref{alg} with the tree $T$ given as input and the sequence of vertex additions $S=(g, \emptyset), (i, g), (l, i), (c, g), (j, i), (k, j), (m, j),$ $(d, c), (b, c), (f, b), (a, b), (e, d), (h, d), (n, h), (o, h)$.}
    \label{fig:alg_example}
\end{figure}

An ordered pair $(p,v)$ is called \emph{critical in an s-model} $\M$ if  

\begin{itemize}
    \item $v = \emptyset$; or
    \item $(p,v) \in E(T)$  and
    \begin{itemize}
        \item $b^{\ell}_{\M}(p,v) = b^{1}_{\M}(p,v)$ and $(v, u^{1}_{\M}(p,v))$ is critical; or
        \item $b^{\ell}_{\M}(p,v) = b^{2}_{\M}(p,v) + 1$ and $(v, u^1_{\M}(p,v)), (v, u^2_{\M}(p,v))$ are critical. 
    \end{itemize}
\end{itemize}
In Figure~\ref{fig: M_1}, $(a, \emptyset)$, $(b,a)$, $(c,b)$, and $(i,j)$ are examples of critical pairs, while $(f,b)$, $(g,i)$, and  $(c,g)$ are not.

Let $(p,v) \in E(T)$ and $T(p,v)$ be the connected component of $T - (N(p) \setminus \{v\})$ that contains $p$. Intuitively, $T(p,v)$ is what remains of $T$ when all neighbors of $p$, except~$v$, are removed along with the vertices that become disconnected from $p$.
\vspace{2mm}
\begin{lem} \label{lem:vcritical}
    Let $T$ be a tree and $\M \in \M(T)$. If $(p,v) \in E(T)$ is a critical pair in $\M$, then $b^{\ell}_T(p,v) = b^{\ell}_{\M}(p,v)$.
\end{lem}
\begin{proof}
 Let $u^i = u^i(p,v)$ and $b^i = b^i(p,v)$, for all $1 \leq i \leq 3$. The proof is by induction on $n = |V(T(p,v))|$. If $n = 2$, then $b^{\ell}_T(p,v) = b^{\ell}_{\M}(p, v) = 0$. Assume $n > 2$ and that the result holds for any such a tree $T(p,v)$ with less than $n$ vertices. As $(p,v) \in E(T)$ is critical, then
\begin{itemize}
    \item if $b^{\ell}_{\M}(p,v) = b^1$ and $(v, u^1)$ is critical, then by induction hypothesis we have that $b^{\ell}_T(p,v) \geq b^{\ell}_T(v,u_1) =  b^{\ell}_{\M}(v, u_1) = b^1$. On the other hand, $b^{\ell}_T(p,v) \leq b^{\ell}_{\M}(p,v) = b^1$. Thus, $b^{\ell}_T(p,v) = b^1 = b^{\ell}_{\M}(p,v)$.
    \item if $b^{\ell}_{\M}(p,v) = b^2 + 1$ and $(v, u^1), (v, u^2)$ are critical, then by induction hypothesis, $b^{\ell}_T(v,u^1) = b^{\ell}_{\M}(v,u^1) = b^1$ and $b^{\ell}_T(v,u^2) = b^{\ell}_{\M}(v,u^2) = b^2$. Since $b^1 \geq b^2$ and $u^1$ or $u^2$ will bend with respect to the edge $(p,v)$, thus $b^{\ell}_T(p,v) \geq b^2+1$. On the other hand, $b^{\ell}_T(p,v) \leq b^{\ell}_{\M}(p,v) = b^2+1$. Thus, $b^{\ell}_T(p,v) = b^2+1 = b^{\ell}_{\M}(p,v)$. 
\end{itemize}
\end{proof}
\vspace{2mm}
\begin{cor} \label{cor:lowerbound}
Let $T$ be a tree and $\M \in \M(T)$. If $(p,v)$ is a critical pair in $\M$, then $b(T) \geq b^{\ell}_{\M}(p,v)$.
\end{cor}
\begin{proof}
By Lemma~\ref{lem:vcritical}, we have that $b(T) \geq b^{\ell}_{T}(p,v) = b^{\ell}_{\M}(p,v)$.
\end{proof}
\vspace{2mm}
\begin{lem} \label{lem:naoaumenta}
    Let $\M'$ be the s-model after the application of the balance step of Algorithm~\ref{alg} on s-model $\M$ in a vertex $v$ which is not balanced.  Then, $b(\M') \leq b(\M)$.
\end{lem}
\begin{proof}
Let $b_{\M}(v)$ denote the maximum number of bends found in leaf-to-leaf paths in $\M$ that contain $v \in V(T)$. Let $b^i = b^i_{\M}(v)$ and $u^i = u^i_{\M}(v)$ for all $1 \leq i \leq 4$. Since $v$ is not balanced in $\M$, then $u^1$ and $u^2$ are not aligned in $\M$. Therefore, $b_{\M}(v) = b^1+b^2+1$. On the other hand, $b_{\M'}(v) = \max\{b^1+b^2, b^1+b^3+1\} \leq b^1+b^2 + 1 = b_{\M}(v)$. After the balance, only the paths that contain $v$ have their number of bends changed. Thus, $b(\M') \leq b(\M)$. 
\end{proof}
\vspace{2mm}
\begin{lem} \label{lem:critico}
    Let $\M'$ be the s-model after the call of \textsc{Balance}($\M$, $p$, $v$) of Algorithm~\ref{alg}, for some $(p,v) \in E(T)$ and $\M \in \M(T)$. If $p  = u^1_{\M}(v)$, then $(p,v)$ is critical in $\M'$.
\end{lem}
\begin{proof}
Let $\M'$ be the s-model being transformed by the call of the procedure and its recursive calls, until its final state when the first procedure call returns. The proof is by induction on $n = |V(T(p,v))|$. If $n = 2$, then $(p,v)$ is trivially critical. Assume $n > 2$ and that the result holds for any such a tree $T(p,v)$ with less than $n$ vertices. The first part of the procedure is to call it recursively for each $(v,w)$, where $w \in N(v) \setminus \{p\}$. Note that, since $\M'$ is equal to $\M$ initially and $p = u^1_{\M}(v)$, we have that $p = u^1_{\M'}(v)$, and therefore $b^\ell_{\M'}(w,v) \geq b^\ell_{\M'}(v,p) \geq b^\ell_{\M'}(v,w) \geq  b^\ell_{\M'}(w, w')$ for all $w' \in N(w) \setminus \{v\}$. Thus, $v  = u^1_{\M'}(w)$. Since $T(v,w)$ has less vertices than $T(p,v)$, we have by induction hypothesis that $(v,w)$ is critical in $\M'$ after the call of \textsc{Balance}($\M'$, $v$, $w$). Let $u^i = u^i_{\M'}(v)$ and $b^i = b^i_{\M'}(v)$, for all $1 \leq i \leq 4$ right before the balance step. Since $b^\ell_{\M}(v,p) = b^\ell_{\M'}(v,p)$, while $b^i_{\M'}(p,v) \leq b^i_{\M}(p,v)$ for all $1 \leq i \leq 3$, then $p = u^1$. Thus, $p$ is aligned with $u^1_{\M'}(p,v)$ and thus $b^{\ell}_{\M'}(p,v) = \max \{b^1_{\M'}(p,v), b^2_{\M'}(p,v)+1\}$. As both $(v, u^1_{\M'}(p,v))$ and $(v, u^2_{\M'}(p,v))$ are critical, then $(p,v)$ is critical.
\end{proof}
\vspace{2mm}
\begin{thm}
Given a tree $T$, let $\M$ be the s-model produced by the execution of Algorithm~\ref{alg} on input $T$. Then, $b(\M) = b(T)$.
\end{thm}
\begin{proof}
The proof is by induction on the number of vertices $n$ of $T$. If $n=1$, then $\M$ consists of just a vertex $v_0$, and $b(\M) = b(T) = 0$. Suppose, the theorem holds for any input tree having less than $n>1$ vertices. Let $T'$ be the tree before the addition of the last vertex $v_{n-1}$ and $\M'$ be the s-model of $T'$ produced so far. By hypothesis induction, $b(\M') = b(T')$. If $b(\M) = b(\M')$, then $b(\M) = b(T') \leq b(T)$ and, thus, $b(\M) = b(T)$. If $b(\M) > b(\M')$, then the number of bends cannot have increased by more than $1$ unit, since the addition of $v_{n-1}$ could have increased the maximum number of bends, but by Lemma~\ref{lem:naoaumenta}, the balance step cannot increase it. Thus, $b(\M) = b(\M')+1$ and all paths that bend the most (say, $k$ bends) have an extreme vertex on $v_{n-1}$. Given a path $P$ with $k$ bends, walk through the vertices of $P$ from $v_{n-1}$ to the other extreme vertex of $P$ and let $w_0, w_1, w_2, \ldots, w_k$ be the vertices such that $w_0 = v_{n-1}$ and the others are those in which $P$ bends, ordered according to such a walk. Let $w'_i$ be the vertex that immediately succeeds $w_i$ in this walk over $P$ (note that $w_k$ must be an internal vertex of $P$, and thus $w'_i$ is well defined for all $0 \leq i \leq k$). We claim that, for all $0 \leq i \leq k$, 
\begin{enumerate}[(i)]
    \item $(w'_i, w_i)$ is critical
    \item $b^{\ell}_{\M}(w'_i,w_i) = i$
\end{enumerate}
and since $b^{\ell}_{\M}(w'_k,w_k) = k$ and $(w'_k,w_k)$ is critical by such a claim, we have by Corollary~\ref{cor:lowerbound} that $b(T) \geq b^{\ell}_{\M}(w'_k,w_k) = k$. Since $b(T) \leq b(\M) = k$, then $b(T) = k = b(\M)$, concluding the proof. 

It remains to prove the claim, which is done by induction on $i$. For $i=0$, we have that $w'_0 = p_{n-1}$ and $( p_{n-1},  v_{n-1})$ is clearly critical; besides, $b^{\ell}_{\M}(p_{n-1},v_{n-1}) = 0$, validating the claim. So, assume $0 < i \leq k$ and the claim holds for all values less than $i$. 

Since $w_i$ is a bend of $P$, $d(w_i) \geq 2$. Moreover, $d(w_i) \neq 2$ or, otherwise, the balance step of the algorithm on the parameter $v = w_i$ would have eliminated such a bend. Therefore, $d(w_i) \geq 3$. Again, since $w_i$ is a bend, and there are no bends in $P$ between $w_i$ and $w_{i-1}$, we have that $b^{\ell}_{\M}(w'_i,w_i) \geq b^{\ell}_{\M}(w'_{i-1},w_{i-1}) + 1$. By induction hypothesis, $b^{\ell}_{\M}(w'_{i-1},w_{i-1}) = i-1$ and, thus, $b^{\ell}_{\M}(w'_i,w_i) \geq i$. On the other hand, as $P$ has exactly $i$ bends from $w_i$ to $w_0$, there is no other way to continue the subpath of $P$ from $w_k$ to $w_i$ up to a leaf in $\M$ using more than $i$ bends in order to not contradict the maximality of $P$ in number of bends. As a consequence, $b^{\ell}_{\M}(w'_i,w_i) \leq i$, proving property (ii) of the claim.

Let $w''_i$ be the vertex that immediately precedes $w_i$ in $P$. Let $z \in N(w_i) \setminus \{w'_i, w''_i\}$. Due to the maximality in the number of bends of $P$, for all $z' \in N(z) \setminus \{w_i\}$, we have that $b^{\ell}_{\M}(z, z') \leq b^{\ell}_{\M}(w_i, z) \leq \min \{b^{\ell}_{\M}(w_i,w'_i), b^{\ell}_{\M}(w_i,w''_i)\} \leq b^{\ell}_{\M}(z, w_i)$ and, thus, $w_i = u^1_{\M}(z)$. Therefore, after the calling of \textsc{Balance}($\M$, $w_i$, $z$), we have that $(w_i, z)$ is critical by Lemma~\ref{lem:critico}. On the other hand, by induction hypothesis, $b^{\ell}_{\M}(w'_{i-1},w_{i-1}) = i-1$ and $(w'_{i-1},w_{i-1})$ is critical. Let $P'$ be the directed path from $w_i$ to $w'_{i-1}$. Since $P'$ has no bends, and since $b^{\ell}_{\M}(w_i,w'') = i-1$, then $b^{\ell}_{\M}(a,b)$ for all $(a,b) \in P'$. As   $(w'_{i-1},w_{i-1})$ is critical, so is $b^{\ell}_{\M}(a,b) = i-1$ for all $(a,b) \in P'$ (in particular, $(w_i,w'')$). As both $(w_i, u^1_{\M'}(w'_i,w_i))$ and $(w_i, u^2_{\M'}(w'_i,w_i))$ are critical, then $(w'_i,w_i)$ is critical, proving property (i) of the claim.
\end{proof}
Therefore, the correctness of the algorithm holds. Now, we turn our attention to some other properties regarding the number of bends of trees.

Consider a graph $G=(V, E)$, $V'\subset V$ and $E'\subset V\times V$. We denote $G[V\setminus V']$ by $G\setminus V'$ and $(V, E\cup E')$ by $G\cup E'$. Let $r\in V(T)$ such that $d(r)=2$, and $\M\in\M(T)$ a balanced s-model of $T$. Let $r_1$ and $r_2$ be the neighbors of $r$ in $T$. We define $T^f_0$ as the tree with a single vertex and, for all $k\geq1$, $T^{f}_{k}$ is the tree rooted in $r$ such that $b^{\ell}_{\M}(r, r_1)=b^{\ell}_{\M}(r, r_2)=k-1$ and the number of vertices of $T$ is the least possible. That is, $$T^{f}_{k}=\argmin_{\text{$T$ is a tree} }\{|V(T)|\mid r\in V(T),\, N(r)=\{r_1, r_2\},\, \M\in\M(T),\text{ $\M$ is balanced},\, b^{\ell}_{\M}(r, r_1)=b^{\ell}_{\M}(r, r_2)=k-1\}\, .$$

A rooted \emph{binary tree} $T$ is an ordered tree with root $r$ in which every vertex has at most two children. The \emph{height} of $T$ is the number of edges in the longest existing path that has as extreme vertices $r$ and a leaf of $T$. An \emph{interior vertex} of $T$ is a vertex that is not a leaf of $T$. Consider $u,v\in V(T)$. The \emph{distance} from $u$ to $v$ is the length of the path connecting $u$ and $v$ in $T$. The \emph{level} of a vertex $v\in V(T)$ is the distance from $v$ to $r$. A \emph{full binary tree} is a binary tree in which every interior vertex has exactly two children and all leaves are at the same (highest) level. We denote by $T^{b}_{k}$ the full binary tree with height $k$.

Consider a path $P=u,\ldots,v$ in a tree $T$ and let $T'$ be a subtree of $T$. We say that $P$ \emph{follows through} $T'$ if, and only if, either $u\in V(T')$ or $v\in V(T')$.
\vspace{2mm}
\begin{lem} \label{lema_Tfk}
For all $k\geq 0$, $T^{f}_{k}\cong T^{b}_{k}$.
\end{lem}
\begin{proof}
By induction on $k$. If $k=0$, then $T^{f}_{0}$ is a tree with a single vertex and, thus, $T^{f}_{0}\cong T^{b}_{0}$. If $k=1$, then $b^{\ell}_{\M}(r, r_1)=b^{\ell}_{\M}(r, r_2)=0$ and, thus, $r_1$ and $r_2$ are leaves of $T^{f}_{1}$. Therefore, $T^{f}_{1}\cong T^{b}_{1}$. Assume that $T^{f}_{k'}\cong T^{b}_{k'}$, for all $1\leq k'<k$. Consider the neighbors of $r$, $r_1$ and $r_2$, in $T^{f}_{k}$. Since $T^{f}_{k}$ has the least number of vertices, we will show that $d(r_1)=3$. If $d(r_1)=1$, then $r_1$ is a leaf of $T^{f}_{k}$ and, thus, $b^{\ell}_{\M}(r, r_1)=0$. Therefore, $k=1$ (contradiction!). If $d(r_1)=2$, there exists $r_{11}\in N(r_1)\setminus\{r\}$, such that $b^{\ell}_{\M}(r, r_1)=b^{\ell}_{\M}(r, r_{11})=k-1$. And then, there exists $T=(T^{f}_{k}\setminus\{r_1\})\cup\{(r, r_{11})\}$, which contradicts the minimality of $|V(T^{f}_{k})|$. If $d(r_1)=4$, let $r_{1i}\in N(r_1)\setminus\{r\}$, for all $1\leq i\leq 3$. Note that the path $P$ with $k-1$ bends, having as extreme vertices $r$ and leaves of $T^{f}_{k}$ and containing $r_1$, must follow through $T(r_1, r_{1i})$, for some $i$. Without loss of generality, assume that $r_{11}$ and $r$ are aligned. If $P$ follows through $T(r_1, r_{11})$, there exists $T=T^{f}_{k}\setminus(V(T(r_1, r_{12}))\cup V(T(r_1, r_{13})))$, which contradicts the minimality of $|V(T^{f}_{k})|$. Else, if $P$ follows through $T(r_1, r_{12})$, there exists $T=T^{f}_{k}-V(T(r_1, r_{13}))$, which contradicts the minimality of $|V(T^{f}_{k})|$ and if $P$ follows through $T(r_1, r_{13})$, there exists $T=T^{f}_{k}-V(T(r_1, r_{12}))$, which, again, contradicts the minimality of $|V(T^{f}_{k})|$. Thus, $d(r_1)=3$. Let $r_{11}, r_{12}\in N(r_1)\setminus\{r\}$, such that $r_{11}$ is aligned with $r$. Note that the path having as extreme vertices $r$ and leaves of $T^{f}_{k}$, which contains $r_1$ and has $k-1$ bends, must follow through $T(r_1, r_{12})$. Otherwise, there would exist $T=T^{f}_{k}-V(T(r_1, r_{12}))$, such that $b^{\ell}_{\M}(r, r_{11})=b^{\ell}_{\M}(r, r_2)=k-1$, which contradicts the minimality of $|V(T^{f}_{k})|$. Thus, $b^{\ell}_{\M}(r_1, r_{12})=k-2$. Notice that $b^{\ell}_{\M}(r_1, r_{11})\geq k-2$. By the minimality of $|V(T^{f}_{k})|$, it follows that $b^{\ell}_{\M}(r_1, r_{11})=k-2$. Consider $T'=T(r, r_1)\setminus\{r\}$. Note that, in $T'$, $d(r_1)=2$ and $b^{\ell}_{\M}(r_1, r_{11})=b^{\ell}_{\M}(r_1, r_{12})=k-2$. By the minimality of $|V(T')|$, $T'\cong T^{f}_{k-1}$. By induction hypothesis, $T'\cong T^{b}_{k-1}$. Similarly, let $T''=T(r, r_2)\setminus\{r\}$. Then $T''\cong T^{b}_{k-1}$. Thus, $T^{f}_{k}=(V(T')\cup V(T'')\cup\{r\}, E(T')\cup E(T'')\cup\{(r, r_1), (r, r_2)\})$. Therefore, $T^{f}_{k}\cong T^{b}_{k}$.
\end{proof}

\begin{cor}
For all $k\geq 0$, $|V(T^{f}_{k})|=2^{k+1}-1$.
\end{cor}
\begin{proof}
By Lemma~\ref{lema_Tfk}, $T^{f}_{k}\cong T^{b}_{k}$. Thus, $|V(T^{f}_{k})|=|V(T^{b}_{k})|=2^{k+1}-1$.
\end{proof}
 
Consider the problem of finding a tree $T$ with $k$ bends such that $|V(T)|$ is minimum. Let $T_k$ be a tree with $k$ bends defined in the following way:
\begin{itemize}
    \item For $k=0$, $T_k$ is a tree with a single vertex.
    \item For $k\geq 1$, consider the path $P=v_0, v_1,\ldots, v_k, v_{k+1}$ with $k$ bends. Therefore, each vertex $v_i$ is a bend of $P$, for all $1\leq i\leq k$. We define $T_k$ built from $P$ in the following way: for each $v_i$, add to $P$ a disjoint copy of the tree $T^{f}_{d_i}$, where $d_i=\min\{i-1, k-i\}$, and add the edge $(v_i, r)$, where $r$ is the root of $T^{f}_{d_i}$. 
\end{itemize}
See in Figure~\ref{fig:trees_Tk} two examples of $T_k$'s. We will show that $|V(T_k)|$ is minimum. Note that $|V(P)|=k+2$, thus $|V(T_k)| = (k+2)+ \sum_{i=1}^{k} |V(T^{f}_{d_i})|$, for $k\geq 1$. Note also that, $d_i=d_j$, for all $i$ and $j$ such that $i+j=k+1$. Therefore, consider the following cases,
\begin{itemize}
    \item[(i)] if $k$ is even,
    \begin{eqnarray*}
        |V(T_k)|&=&(k+2) + 2\cdot\sum_{i=1}^{\frac{k}{2}} |V(T^{f}_{i-1})|\\
        &=&(k+2) + 2\cdot\sum_{i=1}^{\frac{k}{2}} (2^i-1)\\
        &=&2^{\frac{k+4}{2}}-2\, .
    \end{eqnarray*}
    \item[(ii)] if $k$ is odd,
    \begin{eqnarray*}
        |V(T_k)|&=&(k+2) + \left|V\left(T^{f}_{\frac{k-1}{2}}\right)\right|+ 2\cdot\sum_{i=1}^{\frac{k-1}{2}} |V(T^{f}_{i-1})|\\
        &=&3\cdot2^{\frac{k+1}{2}}-2\, .
    \end{eqnarray*}
\end{itemize}
Thus,
    $$|V(T_k)|
    = \left\{\begin{array}{rll}
    2^{\frac{k+4}{2}}-2, & \hbox{if} & \hbox{$k$ is even} \\
    3\cdot2^{\frac{k+1}{2}}-2, & \hbox{if} & \hbox{$k$ is odd}\, .
    \end{array}\right.$$
\vspace{2mm}
\begin{thm} \label{teorema1}
Let $T$ be a tree such that $b(T)=k$. Then, $|V(T)| \geq |V(T_k)|$.
\end{thm}
\begin{proof}
Let $T$ be a tree with a balanced s-model $\M$, such that $b(T)=k$ and $P$ a path with $k$ bends in $T$. Let $v_1, v_2,\ldots, v_k$ be the vertices in which $P$ bends. Since $\M$ is balanced, $d(v_i) \geq 3$ for all $1 \leq i \leq k$. Note that $b^{\ell}_{\M}(v_i, v_{i-1})=i-1$ and $b^{\ell}_{\M}(v_i, v_{i+1})=k-i$, for all $v_i$. We claim that for all $v_i$ there is $w_i\in N(v_i)\setminus\{v_{i-1}, v_{i+1}\}$, such that $b^{\ell}_{\M}(v_i, w_i)\geq \min\{i-1, k-i\}$. Consider the following cases:
\begin{itemize}
    \item[(i)] If $i-1<k-i$, then there is a vertex, $w_i$, aligned with $v_{i+1}$, such that $b^{\ell}_{\M}(v_i, w_i)\geq i-1$. Otherwise, $\M$ would not be balanced.
    \item[(ii)] If $k-i<i-1$, then there is a vertex, $w_i$, aligned with $v_{i-1}$, such that $b^{\ell}_{\M}(v_i, w_i)\geq k-i$. Otherwise, $\M$ would not be balanced.
    \item[(iii)] If $k-i=i-1$, then there must be a vertex, $w_i$, aligned with either $v_{i-1}$ or $v_{i+1}$, such that $b^{\ell}_{\M}(v_i, w_i)\geq k-i$. Otherwise, $\M$ would not be balanced.
\end{itemize}
Thus, consider $w_i\in N(v_i)\setminus\{v_{i-1}, v_{i+1}\}$ and the path $P'$ having as extreme vertices $v_i$ and a leaf $\ell\in T(v_i, w_i)$, which contains $w_i$ and has $d_i$ bends, where $d_i=\min\{i-1, k-i\}$. Walk through the vertices of $P'$ from $v_i$ to $\ell$ and let $w_1$ be the vertex in which $P'$ bends for the first time with respect to such a walk. Then, $d(w_1)\geq 3$. Let $w_1'$ be the vertex that immediately precedes $w_1$ in this walk over $P'$ and let $N(w_1)\setminus \{w_1'\}=\{w_{1j} \mid 1\leq j< d(w_1)\}$. Without loss of generality, assume that $w_{11}$ is aligned with $w_1'$. Since $w_1$ is a bend of $P'$, such a path does not follow through $T(w_1, w_{11})$. Without loss of generality, assume that $P'$ follows through $T(w_1, w_{12})$. Then, $b^{\ell}_{\M}(w_1, w_{12})= d_i-1$. Since $\M$ is balanced, $b^{\ell}_{\M}(w_1, w_{11})\geq d_i-1$. Thus, $|V(T(v_i, w_i)\setminus \{v_i\})|\geq |V(T(w_1', w_1)\setminus \{w_1'\})|\geq |V(T^{f}_{d_i})|$. Notice that $|V(T)|\geq |V(P)|+ \sum_{i=1}^{k} |V(T(v_i, w_i) \setminus \{v_i\} )|$ and $|V(P)|\geq k+2$. Therefore,
\begin{equation*}
    |V(T)| \geq (k+2) + \sum_{i=1}^{k} \left|V\left(T^{f}_{d_i}\right)\right| = |V(T_k)|\, .
\end{equation*}
\end{proof}
\begin{figure}[htb]
    \centering
    \subfigure[][]{\includegraphics[scale=0.3]{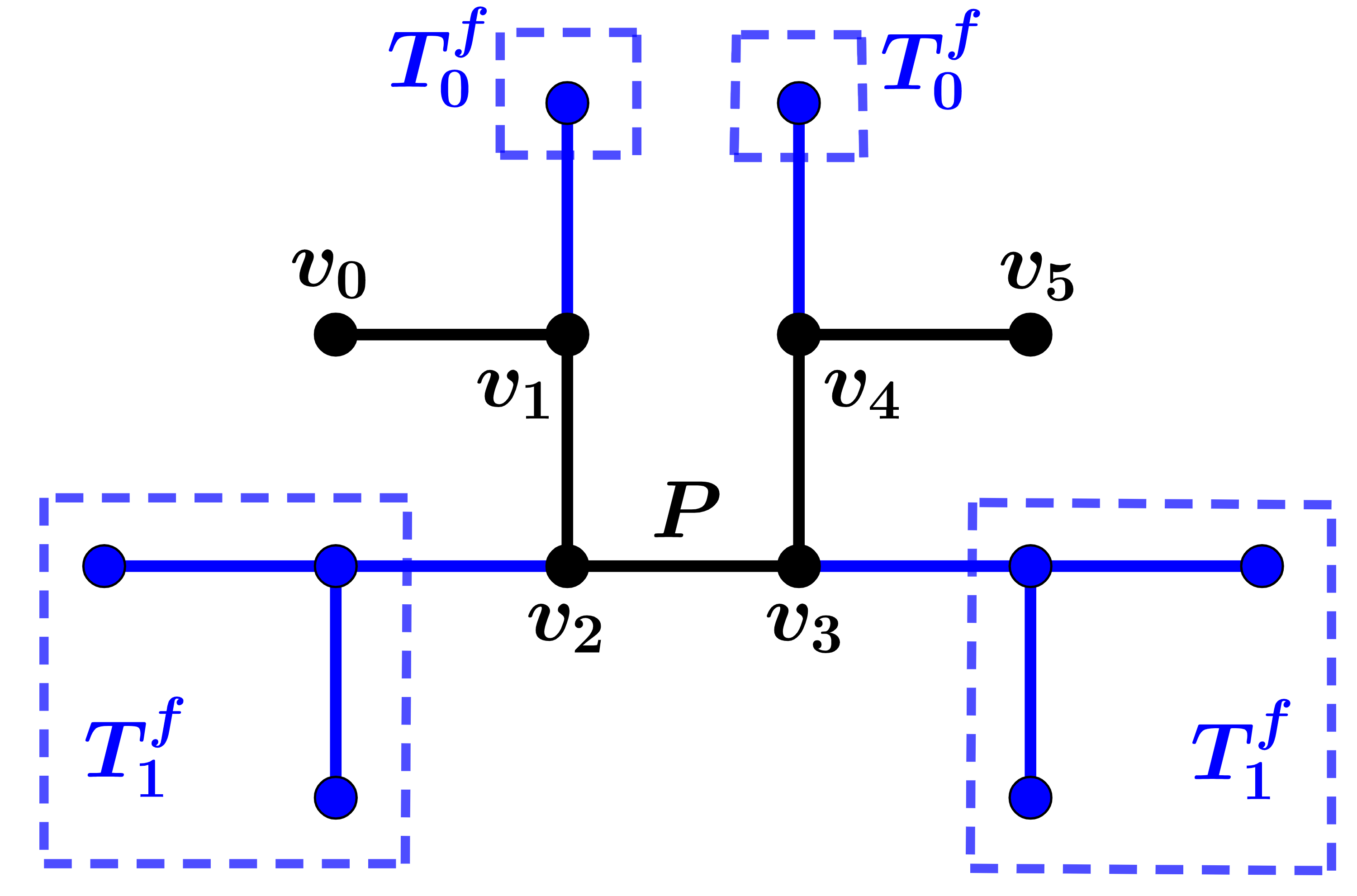} \label{fig:T_4} }
    \quad
    \subfigure[][]{\includegraphics[scale=0.3]{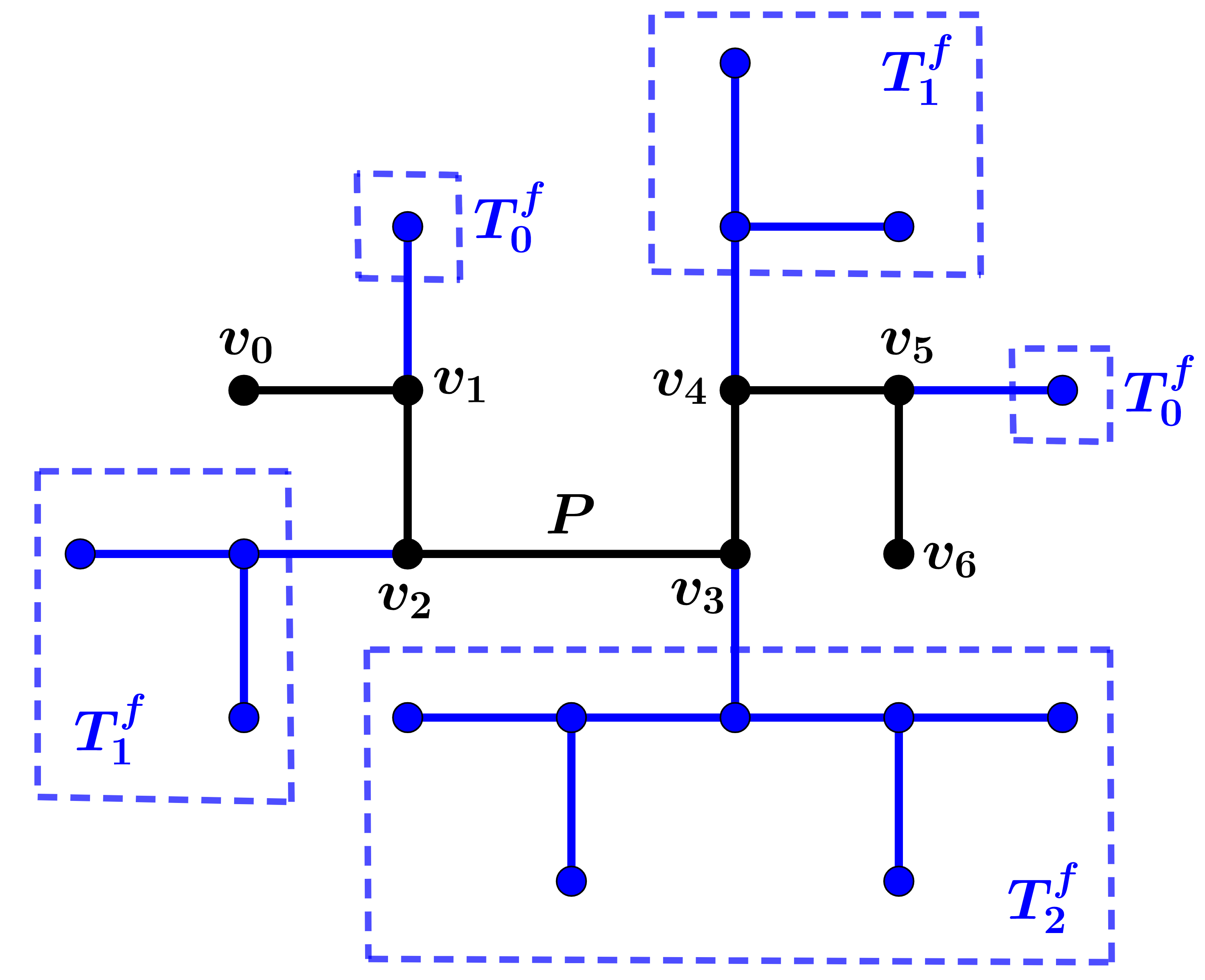} \label{fig:T_5}}
    \caption{Trees $T_4$ (left) and $T_5$ (right).}
    \label{fig:trees_Tk}
\end{figure}
\begin{thm}
If $T$ is a tree with $n$ vertices, then $b(T)\leq 2\cdot\log_{2}(n+2)-4$.
\end{thm}
\begin{proof}
Let $T$ be a tree with $n$ vertices. By Theorem~\ref{teorema1}, $n\geq |V(T_{b(T)})|$. Consider the following cases:
\begin{itemize}
    \item[(i)] If $b(T)$ is even,
    \begin{eqnarray*}
        n &\geq& 2^{\frac{b(T)+4}{2}}-2\\
        b(T) &\leq& 2\cdot\log_{2}(n+2)-4\, .
    \end{eqnarray*}
    \item[(ii)] If $b(T)$ is odd,
    \begin{eqnarray*}
        n &\geq& 3\cdot 2^{\frac{b(T)+1}{2}}-2\\
        b(T) &\leq& 2\cdot\log_{2}\left(\frac{n+2}{3}\right)-1\, .
    \end{eqnarray*}
\end{itemize}
Since $2\cdot\log_{2}(n+2)-4=2\cdot\log_{2}(\frac{n+2}{4})\geq 2\cdot\log_{2}(\frac{n+2}{3\sqrt{2}})=2\cdot\log_{2}(\frac{n+2}{3})-1$ for all $n>0$, then $b(T) \leq 2\cdot\log_{2}(n+2)-4$.
\end{proof}

\section{EPG models of VPT $\cap$ EPT graphs} \label{sec: EPG representations of VPT}

We provide an upper bound on the number of bends of an EPG model of VPT~$\cap$~EPT graphs. The VPT $\cap$ EPT graphs are proved to be those that can be represented in host trees with maximum degree at most $3$:
\begin{thm}[\cite{golumbic1985edge}]
Let $G$ be a graph. The following statements are equivalent:
\begin{itemize}
    \item[(i)] $G$ is both a VPT and an EPT graph.
    \item[(ii)] $G$ has VPT and EPT models on a tree with maximum degree $3$.
\end{itemize}
\end{thm}
In \cite{alcon2015characterizing}, this class is characterized by a family of minimal forbidden induced subgraphs. In \cite{alcon2020b_1}, the authors showed that every Chordal B$_1$-EPG graph is a VPT~$\cap$~EPT graph.

Let $G$ be a $VPT\cap EPT$ graph and $\langle T, \mathcal{P} \rangle$ a VPT model of $G$ in which $\Delta(T)\leq3$. Consider $V(G)=\{v_1, v_2,\ldots, v_n\}$, $V(T)=\{u_1, u_2,\ldots, u_m\}$ and $\mathcal{P}=\{Q_i\mid 1\leq i\leq n\}$. Build an EPG model $\mathcal{R} = \{P_i \mid 1\leq i\leq n\}$ of $G$ in a grid $\mathcal{G}$ in the following way.

First, let $\mathcal{M}$ be a model of $T$ with the minimum number of bends on the grid $\mathcal{G}$, as described in Section~\ref{sec:embedding trees}.
For all edges $e_i\in E(T)$ in $\mathcal{M}$, let $e'_i$ be their midpoints in grid $\mathcal{G}$.
For all $u_i\in V(T)$ such that $d(u_i)=1$, build an auxiliary path, $P'_{u_i}$, going from $u_i$ to~$e'$, where~$e$ is the edge to which $u_i$ is incident.
For all $u_i\in V(T)$ such that $d(u_i)=2$, let $e_1$ and~$e_2$ be the edges incident to $u_i$. Build an auxiliary path  $P'_{u_i}$ having $e'_1$ and $e'_2$ as endpoints.
For all $u_i\in V(T)$ such that $d(u_i)=3$, let $e_1$, $e_2$ and $e_3$ be the edges incident to $u_i$. Note that, at least one of them is vertical and at least one of them is horizontal. Without loss of generality, assume $e_1$ is vertical and $e_2$ is horizontal. Build an auxiliary path  $P'_{u_i}$ having~$e'_1$ and $e'_2$ as endpoints.
For all $Q_i\in\mathcal{P}$, let $u_i$ be an endpoint of $Q_i$. Initialize $P_i$ to be coincident to $Q_i$. Next, consider the following cases:
    \begin{itemize}
        \item[-] if $d(u_i)=2$, enlarge $P_i$ by stretching its endpoint so that it coincides with the endpoint of $P'_{u_i}$ that does not belong to $P_i$ yet.
        \item[-] If $d(u_i)=3$ and $P_i\cap P'_{u_i}=\{u_i\}$, enlarge $P_i$ by stretching its endpoint so that it coincides with the endpoint of $P'_{u_i}$ which does not impose a new bend in $P_i$.
        \item[-] If $d(u_i)=3$ and $P_i\cap P'_{u_i}\neq\{u_i\}$, it implies that $u_i$ is an endpoint of $P_i$ and $P_i$ already contains one of the endpoints of $P'_{u_i}$. In that case, enlarge $P_i$ by stretching its endpoint so that it coincides with the other endpoint of $P'_{u_i}$.
    \end{itemize}
Remove the paths $P'_{u_i}$ for all $1\leq i\leq m$. We will call such a construction a \emph{VPT-EPG transformation}. An important property is that it yields a B$_k$-EPG model of $G$ with $k\leq b(T)$. Indeed, note that if a path $Q_i$ with $u_i$ as an extreme vertex has $b(T)$ bends, then $d(u_i)\leq 2$. Therefore, either $P_i = Q_i$ or $P_i$ is $Q_i$ with their extreme vertices stretched without any new bends. Thus, $P_i$ has $b(T)$ bends and, therefore, $\mathcal{R}$ has a maximum of $b(T)$ bends in any of its paths. Figure~\ref{fig:epg_example_2} presents an EPG model $\mathcal{R} = \{P_i \mid 1\leq i\leq 10\}$ derived for the family $\mathcal{P}=\{Q_i\mid 1\leq i\leq 10\}$ of Figure~\ref{fig:epg_example_1}.
\begin{figure}[htb]
    \centering
    \subfigure[][]{\includegraphics[scale=1]{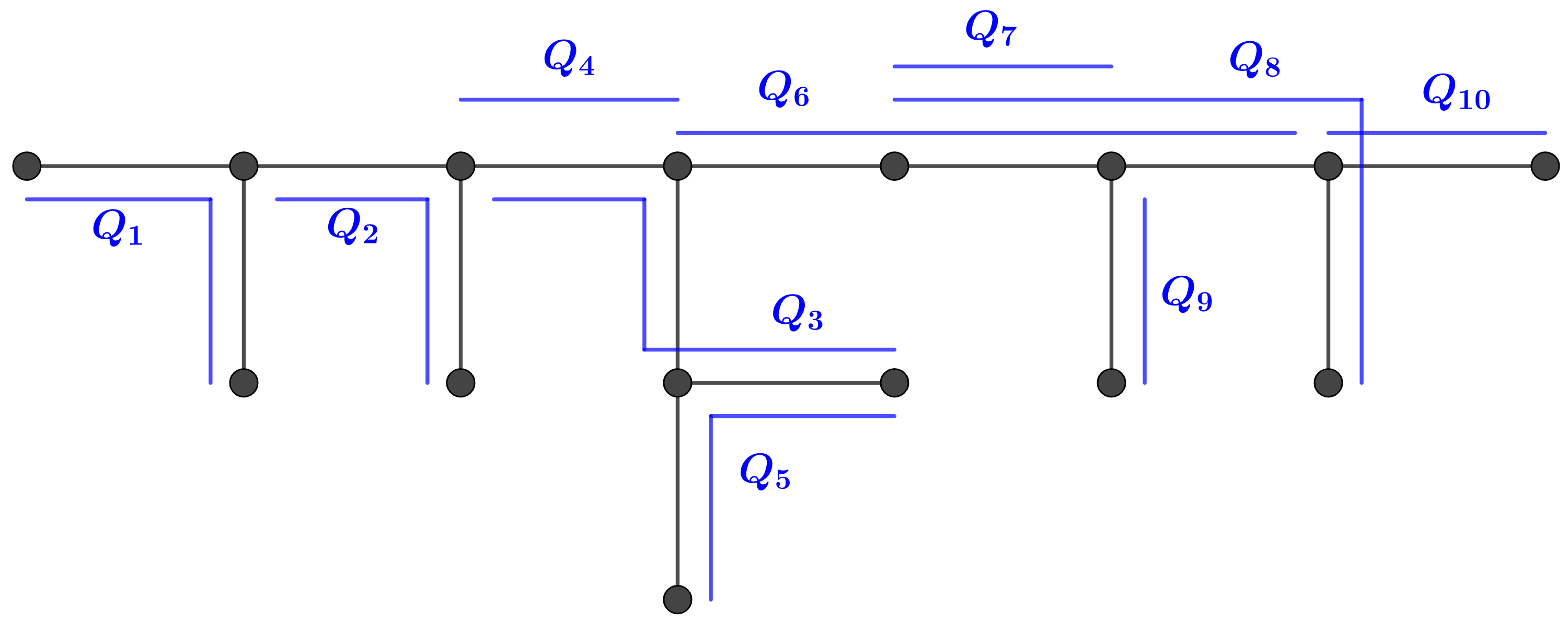} \label{fig:epg_example_1}}
    \quad
    \subfigure[][]{\includegraphics[scale=1]{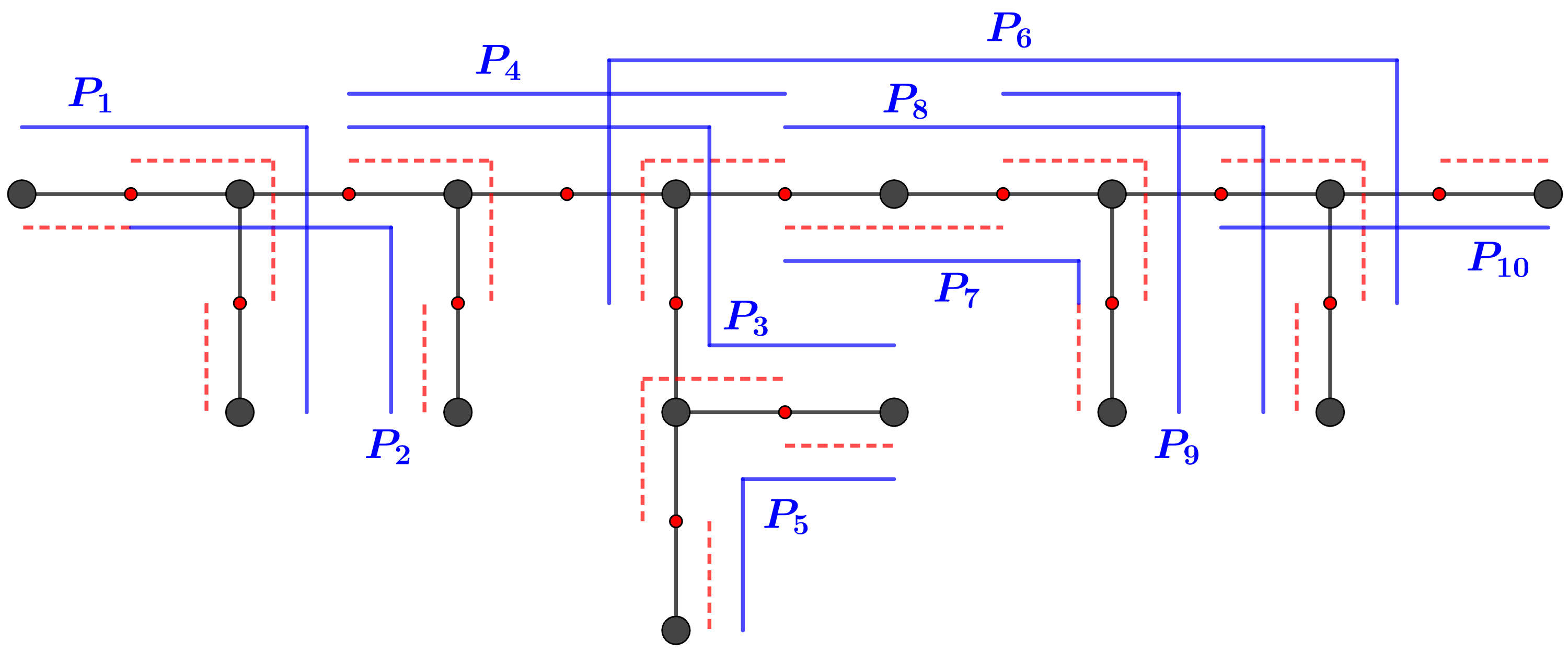} \label{fig:epg_example_2}}
    \caption{Construction of a B$_k$-EPG model with $k\leq  b(T)$.}
    \label{fig:epg_example}
\end{figure}

It is possible to show that, for some examples, the upper bound is tight. However, it may be arbitrarily far from the actual value, as we show next. Consider in $T_k$ the path $P=v_0, v_1,\ldots, v_k, v_{k+1}$ with $k$ bends. Therefore, $v_1,\ldots,v_k$ are the vertices in which $P$ bends. Let $w_i\in N(v_i)\setminus\{v_{i-1}, v_{i+1}\}$. Let $\mathcal{P}_i=\{P_{\ell}=v_i,\ldots,\ell\mid\text{$\ell$ is a leaf of $T(v_i, w_i)$}\}$. Consider the family of paths $\mathcal{P}_k=\bigcup^{k}_{i=1}\mathcal{P}_i\cup P$ and the VPT graph, $G$, having $\langle T_k, \mathcal{P}_k \rangle$ as a model. Note that by applying a VPT-EPG transformation in $\langle T_k, \mathcal{P}_k \rangle$, the resulting model has $k$ bends, since it contains $P$. On the other hand, notice that $G[\mathcal{P}_i]$ is a clique and $P$ is such that the vertex representing it in $G$ is universal. Thus, since it is possible to build an interval model of $G$, where $(0,k)$ is the interval representing $P$ and each element of $\mathcal{P}_i$ is represented by the interval $(i-1,i-\epsilon)$, for some $\epsilon\in\Re$, with $0<\epsilon<1$, $G$ is B$_0$-EPG. See in Figure~\ref{fig:interval_model} an example of such a construction having $T_4$ as host tree.
\begin{figure}[htb]
    \centering
    \subfigure[][]{\includegraphics[scale=0.4]{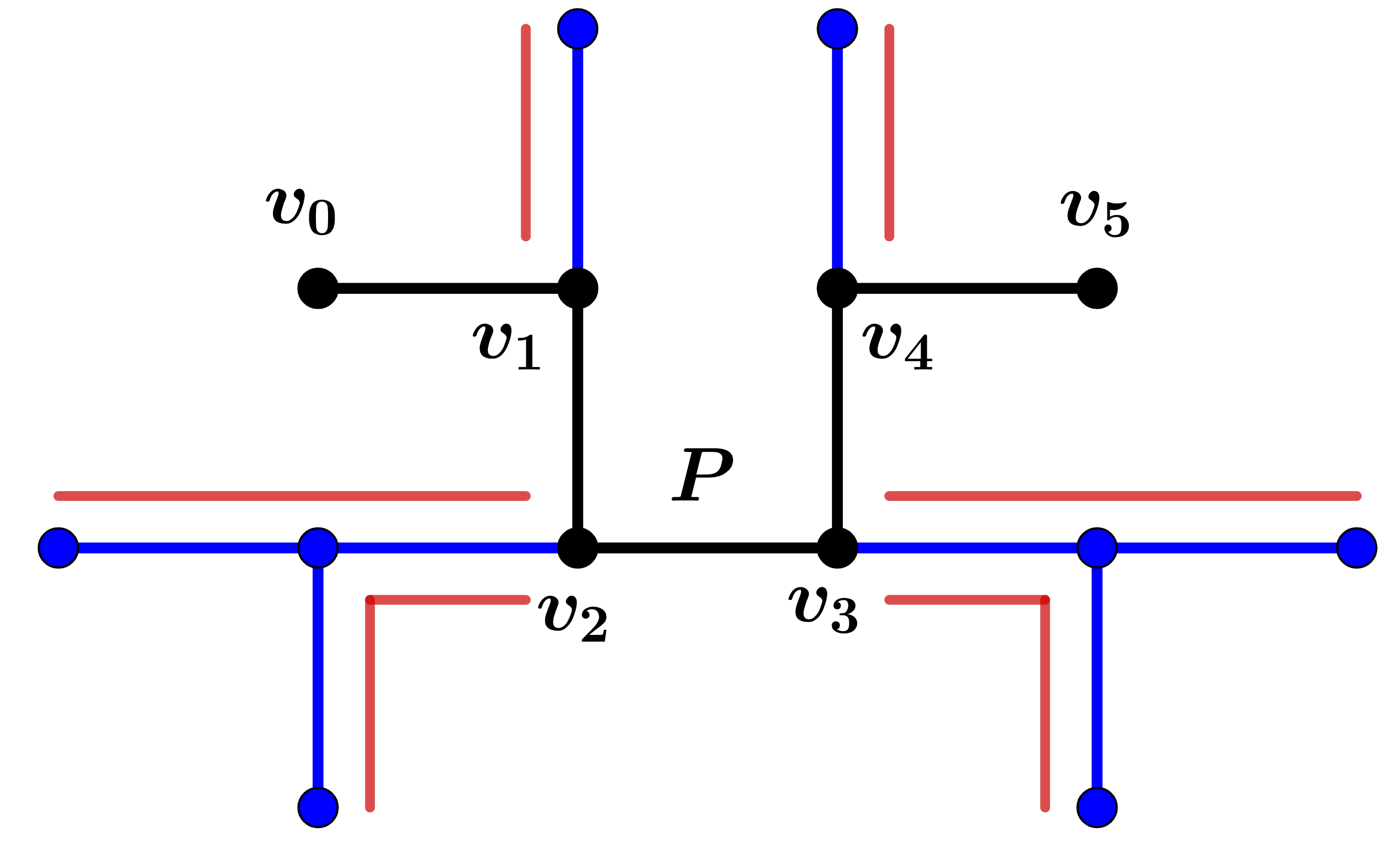} \label{fig:interval_model_0}}
    \quad
    \subfigure[][]{\includegraphics[scale=0.4]{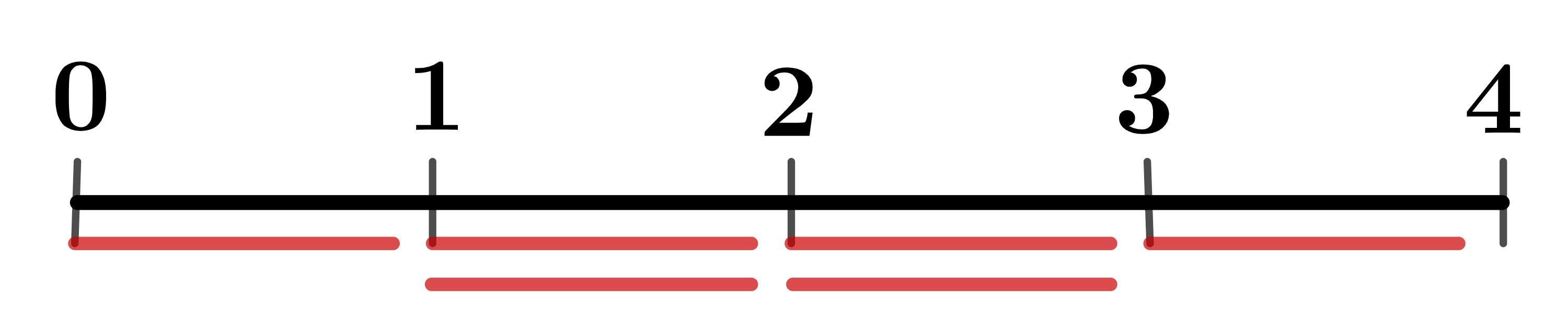} \label{fig:interval_model_1}}
    \caption{The VPT model $\mathcal{P}_4$ (left) and its respective interval model (right).}
    \label{fig:interval_model}
\end{figure}

\section{Conclusion} \label{sec: Conclusion}

In this paper, we presented an algorithm to embed a tree $T$ with $\Delta(T)\leq4$ in a rectangular grid, such that the maximum number of bends over all paths of $T$ is minimized. We also described how to construct s-models with $k$ bends, having the least number of vertices possible for any $k\geq 0$. Moreover, the s-models produced by our algorithm were employed to construct EPG models of VPT $\cap$ EPT graphs providing an upper bound on the number of bends of such graphs. We remark that there are some examples showing that the upper bound is tight.

\bibliographystyle{unsrt}


\end{document}